\documentclass[11pt]{article}

 \usepackage{amscd,amsmath,amsfonts,amssymb,amsthm,cite,mathrsfs,color} 
\usepackage{dsfont} \usepackage{leftidx}
\usepackage{tikz} \usepackage{enumerate}  
 \usepackage{appendix}

 \usepackage[left=2.5cm, right=2.5cm, top=2.5cm, bottom=2.5cm]{geometry}
\usepackage{graphicx} \usepackage{xcolor}
\usepackage[colorlinks,linkcolor=red]{hyperref}
\numberwithin{equation}{section} 

\usepackage{amsthm}

\theoremstyle{plain}
\newtheorem{theorem}{Theorem}[section]

\newtheorem{proposition}[theorem]{Proposition}

\theoremstyle{definition}
\newtheorem{definition}[theorem]{Definition}

\newtheorem{remark}[theorem]{Remark}

\begin{document}
\title{\bf{ 
Repeated erfc   statistics for deformed GinUEs
}}

\author{
Dang-Zheng Liu\footnotemark[1] ~ and     Lu Zhang\footnotemark[2]}
\renewcommand{\thefootnote}{\fnsymbol{footnote}}
\footnotetext[1]{CAS Key Laboratory of Wu Wen-Tsun Mathematics, School of Mathematical Sciences, University of Science and Technology of China, Hefei 230026, P.R.~China. E-mail: dzliu@ustc.edu.cn}
\footnotetext[2]{School of Mathematical Sciences, University of Science and Technology of China, Hefei 230026, P.R.~China. E-mail: zl123456@mail.ustc.edu.cn}


 \maketitle
\begin{abstract}
For    an additive perturbation  of  the complex  Ginibre ensemble   under  a deterministic   matrix $X_0$,   under certain assumption on   $X_0$, 
we  observe  that there are only two kinds of local statistical
patterns at the spectral edge: GinUE statistics and critical statistics, which corresponds to  
regular    or quadratic vanishing spectral points.
 As a continuation of our previous study on critical statistics \cite{LZ23},  in this paper we establish the local statistics of GinUE type at the regular spectral edge, which is characterized by a repeated erfc integral found in \cite{LZ22}.   

 \end{abstract}

\section{Introduction and main results}

\subsection{Introduction}

We  begin our study by specifying the deformed complex Ginibre ensembles (GinUE for short)
 \begin{equation} \label{defmat}X:=X_0 +\sqrt{\frac{\tau}{N}}G,  \end{equation}
where $X_0$   is a deterministic  matrix, and    $G=[g_{ij}]_{i,j=1}^N$ is a random  matrix  whose entries   are i.i.d.  complex normal   variables with mean 0 and variance 1. Or, to be more exact,  we have 
 \begin{definition}  \label{GinU} 
A random  complex  
$N\times N$ matrix  $X$,    is said to belong to the deformed complex  Ginibre ensemble with mean $X_0={\rm diag}\left( A_0,0_{N-r} \right)$ and  time $\tau>0$,  where $A_0$ is a $r\times r$ constant  matrix,  denoted by GinUE$_{N}(A_0)$,  if  the joint probability density function for  matrix entries  is given by   
\begin{equation}\label{model}
P_{N}(A_0;X)=    \Big(\frac{N}{\pi\tau}\Big)^{N^2}\
e^{ -\frac{N}{\tau} {\rm Tr} (X-X_0)(X-X_0)^*}.
\end{equation}

\end{definition}

One  motivation behind the study of the deformed model    \eqref{defmat}   comes mostly
from the general effort toward the understanding of the effect of a perturbation  
on the spectrum of a large-dimensional random matrix; see \cite{BC16} and \cite{LZ22}  for  a more detailed explanation.  This  is a continuation of our previous paper   \cite{LZ23}, where 
a kind of critical statistics has been observed at quadratic vanishing spectral edge  points, and  now focuses on local statistics at the   regular  spectral edge points. 

%


  In the non-perturbative  case,  
 the study of  non-Hermitian random matrices  dates back to the work of    Ginibre  \cite{Gi}  for  i.i.d. real/complex   Gaussian random matrices, and then was extended to  i.i.d.  case.   At a macroscopic  level,     the limit spectral   measure was     the famous \textit{circular law} after  a long list of works   \cite{Ba,Gir,GT,PZ,TV10}; see an excellent  survey \cite{BC12} and references therein. 
  However, at  a   microscopic level,     local  eigenvalue statistics  in the bulk and at the  soft edge  of the spectrum were  revealed    first  for the   Ginibre ensembles   \cite{BS, FH,Ka},  with the help of  exact eigenvalue  correlations.    
  Later  in \cite{TV15},  Tao and Vu  established  a four moment match theorem  in the bulk and at the  soft edge, 
both in the real and complex cases.    Recently,   the four moment
matching condition  has been  removed, so      the edge and bulk  universality have been  proved  for any   random matrix with  i.i.d.  entries,  respectively by 
    Cipolloni,  Erd\H{o}s and  Schr\H{o}der   \cite{CES}   and by Maltsev and Osman  \cite{MO23,Os}. 

 For the deformed model  \eqref{defmat},    the  circular law still holds  under a  small low rank perturbation from $X_0$. see e.g.   \cite[Corollary 1.12]{TV10}. 
 Particularly   
 when  the perturbation strength of $X_0$ is above a certain threshold (supercritical regime),  such a perturbation can  create outliers  outside the unit disk  \cite{Ta13}.  Since then,  this outlier phenomenon has been  extensively  studied  in \cite{BC16,BZ,BR,COW,OR}.  Furthermore,  the fluctuations of  outlier eigenvalues are  investigated respectively by 
 Benaych-Georges and Rochet   for   deformed invariantly matrix ensembles \cite{BR},  and    by Bordenave and Capitaine for deformed i.i.d. random matrices \cite{BC16}.   These fluctuations,  due to non-Hermitian structure,  become much more complicated, and indeed  highly depend on the shape of the Jordan canonical form of the perturbation.   Meanwhile, when  the perturbation strength  is near or below   a certain threshold (subcritical and critical  regimes),   the  finite-rank  perturbation effect  has been  investigated  in our previous paper \cite{LZ22}and  it has been proved that 
 the edge statistics  forms      a  determinantal point processes  with   correlation kernels   characterized  by repeated  integrals of the erfc function,   depending  only on geometric multiplicity of eigenvalue.   
 Together,  all  these  results  establish     a non-Hermitian analogue of the BBP phase  transition for the largest eigenvalue of spiked random matrices \cite{BBP}.

 Our main goal is to   identify all the possible local  eigenvalue statistics at the edge and in the bulk  of the spectrum, for the deformed  model   GinUE$_{N}(A_0)$ defined in \ref{defmat}    under certain restriction \eqref{A0 form}  on $A_0$. 
Related to  this is the limit  spectral measure of the deformed model. Its characterization can be given by  using tools from free probability,    the Brown measure of free circular Brownian motion  $x_0 + \sqrt{\tau}c$, where  $c$  is a  a circular variable,   freely independent of 
a general  operator  $x_0$.    The Brown measure and corresponding density formula has been extensively studied in   \cite{BYZ, BCC13,  EJ,  HZ23, Zh21}.  The exact density  will  play a crucial   role  in the study of local eigenvalue statistics for non-Hermitian random matrices.   Recently,  \cite{EJ}  Erd\'{o}s and Ji  further  observe  a remarkable phenomenon:     
{\textit{“The density of the Brown measure has one of the following two types of behavior
around each point on the boundary of its support-either (i) sharp cut,   or (ii) quadratic decay at certain critical points on the boundary’’}}. This dichotomy can also be observed in the deformed model  \eqref{defmat}  when $X_0$ is  normal under certain restriction   on $A_0$.
 In the previous paper \cite{LZ23} we have proved  that the critical edge point belongs to the quadratic edge and  in this paper we are devoted  to  the regular (non-critical) edge point belongs to the sharp edge.

\subsection{Main results}

Throughout the present paper, we select $A_0$ as a normal matrix with the following form
\begin{equation}\label{A0 form}
A_0={\rm diag}\left(a_1\mathbb{I}_{r_1},\cdots,a_t\mathbb{I}_{r_t},
z_0\mathbb{I}_{r_0},A_{t+1}
\right),
\end{equation} where $t$ is a given non-negative integer,  $z_0$ is  a spectral parameter and   $A_{t+1}$ is a $r_{t+1}\times r_{t+1}$ normal matrix with $z_0$  not an  eigenvalue.
 Put $R_0=N-r$ where  $r=\sum_{\alpha=0}^{t+1} r_{\alpha}$, we assume that    
 $r_0$, $r_{t+1}$ and $R_0$ are  finite, independent of $N$, and  as $N\to \infty$
 \begin{equation}\label{ralpha N}
r_{\alpha}=c_{\alpha}N+R_{\alpha,N},
\quad R_{\alpha,N}=O(1), \quad \alpha=1,\cdots,t.
\end{equation}
  Obviously, 
\begin{equation}\label{calpha sum}
\sum_{\alpha=1}^t  c_{\alpha}=1.
\end{equation}

For given complex numbers $a_1,\cdots,a_t$, introduce a probability measure on the complex plane 
\begin{equation}\label{nu}
d\nu(z)=\sum_{\alpha=1}^t c_{\alpha} \delta(z-a_{\alpha}),
\end{equation}
and 
several relevant  expectations 
\begin{equation}\label{parameter}
P_{00}(z_0):= \int  \frac{1}{|z-z_0|^2} d\nu(z),   \quad P_0:=P_{0}(z_0)=  \int  \frac{z-z_0}{|z-z_0|^4} d\nu(z),
 \end{equation}
and 
\begin{equation}\label{parameter2}
  P_1:=P_{1}(z_0)=  \int  \frac{1}{|z-z_0|^4} d\nu(z).
\end{equation}
These  quantities  are  crucial  to characterize   exact functional form  of  the  limit spectral measure and  to distinguish spectral   properties.  For instances, 
the support of limit spectral measure $\mu_{\infty}$ is
\begin{equation}\label{Support}
\mathrm{Supp}(\mu_{\infty}):=\Big\{z_0\in \mathbb{C}: P_{00}(z_0)\geq \frac{1}{\tau} \Big\};
\end{equation}
with  the boundary curve  characterized by the equation $P_{00}(z_0)= =1/\tau$;   see e.g.  \cite[Proposition 1.2]{BC16}.  On the other hand,   any point on the boundary curve that  is critical  (quadratic) or  not (regular)  edge  the boundary points    is   completely  characterized by   the relation $P_{0}(z_0) \neq 0$ or not.

We also need to define a function,   depending on a real parameter  $n>-1$, by
\begin{equation} \label{IEF}
\mathrm{IE}_{n}(z)= \frac{1}{\sqrt{2\pi} \Gamma(n+1)}
\int_0^{\infty}v^n 
e^{  
-\frac{1}{2}(
v+z)^2}{\rm d}v.
\end{equation}
while  as a limit  of    $n>-1$ from  above 
\begin{equation*} 
\mathrm{IE}_{-1}(z)= \frac{1}{\sqrt{2\pi}  }
e^{  -\frac{1}{2}z^2}.
\end{equation*}
At the regular edge point of the spectrum, a kernel emerges and has been introduced in  \cite{LZ22}  via \begin{equation} \label{correkernel}
K_{n}\left(z,w \right)=\sqrt{\frac{2}{\pi}}\Gamma(n+1) e^{ \frac{1}{2}(z + \overline{ w})^2} \sqrt{
\mathrm{IE}_{n-1}( - z - \overline{z})
\mathrm{IE}_{n-1}( -w- \overline{w})
 }  \,\mathrm{IE}_{n}\big({z}+\overline{w}\big).
\end{equation}
In the special case of $n=0$, $\mathrm{IE}_{0}(z)$ reduces to    the erfc function and the above kernel is   none other than the well-known GinUE edge kernel.    

Now, we are ready to  formulate  the  main result  concerning edge   limits of   
the $n$-point correlation functions  of eigenvalues, denoted by  $R_N^{(n)}(
A_0; z_1,\ldots,z_n)$  as shown in \cite{LZ22}, for the $ {\mathrm{GinUE}}_{N}(A_0)$ ensemble.


\begin{theorem}\label{2-complex-correlation}
For the $ {\mathrm{GinUE}}_{N}(A_0)$ ensemble with $R_0\geq n$ and the assumptions  
\eqref{A0 form} on $A_0$,
 if  $z_0$ is a regular edge point that satisfies
 \begin{equation}  P_{00}(z_0)=   1/\tau \quad \mathrm{and} \quad P_{0}(z_0) \neq 0,
\end{equation}
 then 
as $N\to \infty$ the scaled     correlation functions \begin{small}
\begin{multline} \label{corre2edgecomplex noncritical}
\Big(\frac{\sqrt{P_1}}{|P_0|\sqrt{N}}\Big)^{2n}
R_N^{(n)}\Big(
A_0;z_0-\frac{\sqrt{P_1}\hat{z}_1}{\overline{P_0}\sqrt{N}}
,\cdots,
z_0-\frac{\sqrt{P_1}\hat{z}_n}{\overline{P_0}\sqrt{N}}
\Big)= 
\det\left(
\big[
K_{\widetilde{r}_0}\left( \hat{z}_i, \hat{z}_j \big)
\right]_{i,j=1}^n
\right)
+O\big( N^{-\frac{1}{4}} \big), 
\end{multline}
\end{small}
 hold uniformly for   all 
$\hat{z}_{1}, \ldots, \hat{z}_{n} $
in a compact subset of $\mathbb{C}$,
where  $\widetilde{r}_0= 
 r_0 +R_0$ if $z_0=0,$  and  $\widetilde{r}_0= 
 r_0$ otherwise. 
\end{theorem}

         \begin{remark}  The  microscopic    limit is completely   governed  by  the GinUE edge statistics    in the absence  of  the finite-rank perturbation effect, so  the combination of the critical statistics  proved in \cite{LZ23} probably shows that there are only two types of  statistics at the edge of the spectrum. 
    Besides, although we  assume that the limit spectral measure   $\nu$   of  the normal matrix $A_0$  takes a special form as in  \eqref{nu},   it is expected that  the edge universality  holds for   the general $\nu$.
\end{remark}

{\bf Structure of this paper.}  Section 2  is devoted to the proof of
Theorem \ref{2-complex-correlation}: $z_0\not=0$, and Section 3 is for the postponed proofs of two propositions and also for the proof of Theorem \ref{2-complex-correlation}: $z_0=0$.

\section{Proof of Theorem \ref{2-complex-correlation}}  \label{proofs}

\subsection{Notation and matrix variables} \label{sectnotation}

  By convention,  the conjugate, transpose,  and  conjugate transpose of a complex matrix $A$  are denoted by  $\overline{A}$,  $A^t$ and $A^*$,  respectively.   Besides, the following symbols may be used in this and subsequent sections.

\begin{itemize}

\item   {\bf Tensor}
 The tensor product, or Kronecker product,  of an $m\times n$  matrix  $A=[a_{i,j}]$  and  a $p\times q$ matrix  $B$ is a block matrix $$
A \otimes B=\left[\begin{smallmatrix}
a_{11}B   &\cdots  &  a_{1n} B   \\ \vdots  & \ddots &   \vdots   \\
a_{m1}B   &\cdots  &  a_{mn} B   
\end{smallmatrix}\right].$$
When $m=n$ and $p=q$,  $A \otimes B$ is similar to  $B \otimes A$ by a permutation matrix.

  \item   {\bf  Norm}
  For a complex matrix $M$,  define the Hilbert-Schmidt norm  as $\|M\|:=\sqrt{{\rm{Tr}}(MM^*)}$.     $M=O(A_N)$ for some (matrix) sequence $A_N$  means that each  element  has the same order as  that     of $ A_N$. 
  
        \item   {\bf Constants} $\gamma_{N}=N/(N-n)$, and $f_{\alpha}=|z_0-a_{\alpha}|^2$ ($\alpha=1,\cdots,t$).

\end{itemize}

In view of \eqref{A0 form}, we need to introduce several matrix variables $Y$, $Q_0$, $Q_{t+1}$ and upper triangular $T_{\alpha}$ with non-negative diagonal elements for $\alpha=1,\cdots,t$, which are of sizes $n\times n$, $n\times r_0$, $n\times r_{t+1}$ and $n\times n$ respectively. Introduce scaled variables
\begin{equation}\label{zi}
Z=z_0 I_{n}+N^{-\frac{1}{2}}\hat{Z}, \quad Z=\mathrm{diag}(z_1, \ldots,z_n), \  \hat{Z}=\mathrm{diag}(\hat{z}_1, \ldots,\hat{z}_n).
\end{equation} 
Meanwhile, for $\alpha=1,\cdots,t$,  setting  block matrices as follows 
\begin{equation}\label{A alpha}
A_{\alpha}=
\begin{bmatrix}
\sqrt{\gamma_N}(Z-a_{\alpha}\mathbb{I}_n)  &
-Y^* \\
Y & \sqrt{\gamma_N}
\big(Z^*-\overline{a}_{\alpha}\mathbb{I}_n
\big)
\end{bmatrix},
\end{equation}
\begin{equation}\label{L1hat}
\widehat{L}_1={\rm diag}\big(
A_1\otimes \mathbb{I}_n,\cdots,
A_t\otimes \mathbb{I}_n,
\widehat{B}_{t+1}
\big),
\end{equation}
\begin{equation}\label{B0hat}
\widehat{B}_{t+1}=
\begin{bmatrix}
\sqrt{\gamma_N}
\big(
Z\otimes \mathbb{I}_{\widetilde{r}_{t+1}}
-\mathbb{I}_n\otimes \widetilde{A}_{t+1}
\big)  &
-Y^*\otimes \mathbb{I}_{\widetilde{r}_{t+1}} \\
 Y\otimes \mathbb{I}_{\widetilde{r}_{t+1}} &  \sqrt{\gamma_N}
\big(
Z^*\otimes \mathbb{I}_{\widetilde{r}_{t+1}}
-\mathbb{I}_n\otimes \widetilde{A}_{t+1}^*
\big)
\end{bmatrix},
\end{equation}
\begin{small}
\begin{equation}\label{L0hat}
\widehat{L}_2=
\left[
\begin{smallmatrix}
\left[
\left[\begin{smallmatrix}
a_{\alpha}\mathbb{I}_n & \\ & \overline{a}_{\alpha}\mathbb{I}_n
\end{smallmatrix}\right]
\otimes
(T_{\alpha}^*T_{\beta}) \right]_{\alpha,\beta=1}^t
&
\left[
\left[\begin{smallmatrix}
a_{\alpha}\mathbb{I}_n & \\ & \overline{a}_{\alpha}\mathbb{I}_n
\end{smallmatrix}\right]
\otimes
(T_{\alpha}^*\widetilde{Q}_{t+1}) \right]_{\alpha=1}^t
   \\
\left[
\begin{smallmatrix}
\mathbb{I}_n \otimes 
(\widetilde{A}_{t+1}\widetilde{Q}_{t+1}^*
T_{\beta})  &  \\
 &  
\mathbb{I}_n  \otimes 
 (\widetilde{A}_{t+1}^*\widetilde{Q}_{t+1}^*
T_{\beta})
\end{smallmatrix}\right]_{\beta=1}^t
&
\left[\begin{smallmatrix}
\mathbb{I}_n  \otimes 
(\widetilde{A}_{t+1}\widetilde{Q}_{t+1}^*\widetilde{Q}_{t+1})
 &  \\
& 
\mathbb{I}_n  \otimes 
(\widetilde{A}_{t+1}^*\widetilde{Q}_{t+1}^*\widetilde{Q}_{t+1})
\end{smallmatrix}\right]
\end{smallmatrix}\right],
\end{equation}
\end{small}
where
\begin{equation}\label{Q convenience} 
\widetilde{Q}_{t+1}=\left[ Q_0,Q_{t+1}
\right],\quad
\widetilde{A}_{t+1}={\rm diag}\left( z_0\mathbb{I}_{r_0},A_{t+1}
\right),\quad
\widetilde{r}_{t+1}=r_0+r_{t+1},
\end{equation}
whenever $r_0>0$.


 On the other hand, rewrite $T_{\alpha}$ as  sum of  a diagonal matrix $\sqrt{T_{{\rm d},\alpha}}$ and a strictly upper triangular  matrix ${T_{{\rm u},\alpha}}$
\begin{equation}\label{Talpha 2}
T_{\alpha}=\sqrt{T_{{\rm d},\alpha}}+T_{{\rm u},\alpha}:=\begin{bmatrix}
\sqrt{t_{1,1}^{(\alpha)}} & \cdots & t_{i,j}^{(\alpha)} \\
 & \ddots & \vdots \\
 & & \sqrt{t_{n,n}^{(\alpha)}}
\end{bmatrix},
\end{equation} 
where
\begin{equation}\label{Tdalpha}
T_{{\rm d},\alpha}={\rm diag}
\big(
t_{1,1}^{(\alpha)},\cdots,t_{n,n}^{(\alpha)}
\big).
\end{equation}

\subsection{Concentration reduction} \label{sectsketch}

Combining integral representations for the $n$-point correlations given in our previous paper \cite[Proposition 1.3]{LZ22} and duality formulas for auto-correlation functions of characteristic polynomials in \cite{Gr, LZ22}, we can write the $n$-point correlation function for  the GinUE$_{N}(A_0)$ ensemble in Definition \ref{GinU} in an appropriate form which is ready for asymptotic analysis by the Laplace method argument. This has already been done in \cite[Proposition 2.2]{LZ23}, but for convenience we restate it as follows. 
\begin{proposition}\label{foranlysis}
For the $ {\mathrm{GinUE}}_{N}(A_0)$ ensemble with $R_0\geq n$,  with the assumptions  
\eqref{A0 form} on $A_0$ and the above notations,  
 the $n$-point correlation function  can be written as \begin{equation}
\begin{aligned} \label{intrep}
R_N^{(n)}(A_0;z_1,\cdots,z_n)=
\widetilde{C}_{N,n}
\int \cdots \int g(T,Y,Q)\exp\{ Nf(T,Y,Q) \}
{\rm d}Y
{\rm d}Q_0{\rm d}Q_{t+1} \prod_{\alpha=1}^t
{\rm d}T_{\alpha},
\end{aligned}
\end{equation}
where 
\begin{small}
\begin{equation}\label{gYUT}
\begin{aligned}
g(T,Y,Q)=& \Big(\det\big(
\mathbb{I}_n-\sum_{\alpha=1}^t T_{\alpha}T_{\alpha}^*
-Q_0Q_0^*-Q_{t+1}Q_{t+1}^*
\big)\Big)^{R_0-n}  \Bigg(
\det\!\begin{bmatrix}
\sqrt{\frac{N}{N-n}}Z & -Y^*   \\
Y  &   \sqrt{\frac{N}{N-n}}Z^*
\end{bmatrix}
\Bigg)^{R_0-n}
\\
&\times
\Big(\prod_{\alpha=1}^t\prod_{j=1}^n
\big(
t_{j,j}^{(\alpha)}
\big)^{R_{\alpha,N}-j}\Big)
\Big(\prod_{\alpha=1}^t
\big(
\det(A_{\alpha})
\big)^{R_{\alpha,N}-n} \Big) \det\!\bigg(
\widehat{L}_1+\sqrt{\frac{N}{N-n}}\widehat{L}_2
\bigg),
\end{aligned}
\end{equation}
\end{small}
and
\begin{small}
\begin{equation}\label{fTY}
\begin{aligned}
&f(T,Y,Q)=\sum_{\alpha=1}^t
c_{\alpha}\log\det(T_{\alpha}T_{\alpha}^*) 
+\frac{1}{\tau}h(T,Q)
-
\frac{N-n}{\tau N}
{\rm Tr}\big(
YY^*
\big)+\sum_{\alpha=1}^tc_{\alpha}\log\det(A_{\alpha}),
\end{aligned}
\end{equation}
\end{small}
with
\begin{small}
\begin{equation}\label{hQrewrite}
\begin{aligned}
&h(T,Q)=\sum_{\alpha=1}^t\left(-|z_0-a_{\alpha}|^2{\rm Tr}\big(
T_{\alpha}T_{\alpha}^*
\big)
+N^{-\frac{1}{2}}
\left(
\overline{a}_{\alpha}{\rm Tr}\big(
\hat{Z} T_{\alpha}T_{\alpha}^*
\big)+a_{\alpha}{\rm Tr}\big(
\hat{Z}^* T_{\alpha}T_{\alpha}^*
\big)
\right) 
\right)        \\
&+
|z_0|^2\Big(
\sum_{\alpha=1}^t
{\rm Tr}\big(
T_{\alpha}T_{\alpha}^*
\big)+{\rm Tr}\big(
Q_0Q_0^*
\big)+{\rm Tr}\big(
Q_{t+1}Q_{t+1}^*
\big)
\Big)         \\       
&+N^{-\frac{1}{2}}
\Big(
\overline{z}_0{\rm Tr}\big(
\hat{Z} Q_0Q_0^*
\big)+z_0{\rm Tr}\big(
\hat{Z}^* Q_0Q_0^*
\big)
+
{\rm Tr}\big(
\hat{Z} Q_{t+1}A_{t+1}^*Q_{t+1}^*
\big)+{\rm Tr}\big(
\hat{Z}^* Q_{t+1}A_{t+1}Q_{t+1}^*
\big)
\Big)  \\
&+\sum_{1\leq i < j\leq n}
\Big |
\Big(
\sum_{\alpha=1}^t a_{\alpha}T_{\alpha}T_{\alpha}^*
+z_0Q_0Q_0^*+Q_{t+1}A_{t+1}Q_{t+1}^*
\Big)_{i,j}
\Big|^2          \\
&-{\rm Tr}Q_{t+1}
\big(
z_0\mathbb{I}_{r_{t+1}}-A_{t+1}
\big)^*
\big(
z_0\mathbb{I}_{r_{t+1}}-A_{t+1}
\big)
Q_{t+1}^*. 
\end{aligned}
\end{equation}
\end{small}
Here 
\begin{equation} \label{norm-1}
\widetilde{C}_{N,n,\tau}=
\Big( \frac{N-n}{\pi\tau} \Big)^{n^2}
\frac{1}{C_{N,\tau}}e^{-\frac{N}{\tau}\sum_{k=1}^n|z_k|^2}
\prod_{1\leq i <j \leq n}|z_i-z_j|^2
\prod_{\alpha=1}^t
\frac{\pi^{nr_{\alpha}-\frac{n(n-1)}{2}}}
{\prod_{j=1}^n(r_{\alpha}-j)!},
\end{equation}
\begin{equation*}
C_{N,\tau}= 
\tau^{nN-\frac{n(n-1)}{2}}
\pi^{n(r+1)} N^{-\frac{1}{2}n(n+1)} (N-n)^{-n(N-n)}     \prod_{k=N-n-r}^{N-r-1}  k !,
\end{equation*}
and    matrix  variables   are restricted to 
the integration region 
\begin{equation}\label{integrationregionprop}
\sum_{\alpha=1}^t T_{\alpha}T_{\alpha}^*+Q_0Q_0^*+Q_{t+1}Q_{t+1}^*\leq \mathbb{I}_n.
\end{equation}
\end{proposition}


In order to  find the exact approximation of the integral 
representation  in \eqref{intrep},  we first give   a concentration reduction via restricting the integration region to a small domain.   For this, 
 it's easy to see from  \eqref{A alpha} that  
\begin{equation*}
A_{\alpha}=
\begin{bmatrix}
\sqrt{ \gamma_N} (z_0-a_{\alpha}) \mathbb{I}_n &
-Y^* \\
Y &
\sqrt{ \gamma_N} \overline{z_0-a_{\alpha}}\mathbb{I}_n
\end{bmatrix}+
\sqrt{ \gamma_N}N^{-\frac{1}{2}}
\begin{bmatrix}
\hat{Z}  &   \\
&   \hat{Z}^*
\end{bmatrix},
\end{equation*}
and 
\begin{small}
\begin{equation}\label{logAalpha decompose}
\log\det(A_{\alpha})=\log\det\left(
\gamma_N f_{\alpha}\mathbb{I}_n+YY^*
\right)+\log\det\left(
\mathbb{I}_{2n}+\sqrt{ \gamma_N}N^{-\frac{1}{2}}
\widehat{A}_{\alpha}
\right),
\end{equation}
\end{small}
where 
\begin{equation}\label{falpha}f_{\alpha}:=|z_0-a_{\alpha}|^2, \quad \alpha=1,\cdots,t,\end{equation}
and
\begin{equation}\label{HAalpha}
\begin{aligned}
&\widehat{A}_{\alpha}=
\left[\begin{smallmatrix}
\sqrt{\gamma_N}\overline{z_0-a_{\alpha}}\left(
 \gamma_N f_{\alpha}\mathbb{I}_n+Y^*Y
\right)^{-1}\hat{Z}
& Y^*\left(
\gamma_N f_{\alpha}\mathbb{I}_n+YY^*
\right)^{-1}\hat{Z}^*       \\
-Y\left(
\gamma_N f_{\alpha}\mathbb{I}_n+Y^*Y
\right)^{-1}\hat{Z}
& \sqrt{\gamma_N}(z_0-a_{\alpha})\left(
\gamma_Nf_{\alpha}\mathbb{I}_n+YY^*
\right)^{-1}\hat{Z}^*
\end{smallmatrix}\right].
\end{aligned}
\end{equation}

Take all the  leading terms of $f(T,Y,Q)$ in  \eqref{fTY} and put  
\begin{small}
\begin{equation}\label{f0TY}
\begin{aligned}
&f_0(\tau;T,Y,Q)=\sum_{\alpha=1}^t\left(
c_{\alpha}\log\det(T_{\alpha}T_{\alpha}^*) 
+c_{\alpha}
\log\det\big(
f_{\alpha}\mathbb{I}_n+YY^*
\big)\right)       
 +
\frac{1}{\tau}\Big(h_{0}(T,Q)- {\rm Tr}\big(YY^*)\Big)
\end{aligned}
\end{equation}
\end{small}
with \begin{small}
\begin{equation*}
\begin{aligned}
&h_{0}(T,Q)= 
-\sum_{\alpha=1}^t f_{\alpha}{\rm Tr}\big(
T_{\alpha}T_{\alpha}^*
\big)
-{\rm Tr}Q_{t+1}
\big(
z_0\mathbb{I}_{r_{t+1}}-A_{t+1}
\big)^*
\big(
z_0\mathbb{I}_{r_{t+1}}-A_{t+1}
\big)
Q_{t+1}^*               \\
&+
|z_0|^2\bigg(
\sum_{\alpha=1}^t
{\rm Tr}\big(
T_{\alpha}T_{\alpha}^*
\big)+{\rm Tr}\big(
Q_0Q_0^*
\big)+{\rm Tr}\big(
Q_{t+1}Q_{t+1}^*
\big)
\bigg)         \\ 
&+\sum_{1\leq i < j\leq n}
\bigg|
\bigg(
\sum_{\alpha=1}^t a_{\alpha}T_{\alpha}T_{\alpha}^*
+z_0Q_0Q_0^*+Q_{t+1}A_{t+1}Q_{t+1}^*
\bigg)_{i,j}
\bigg|^2.         
\end{aligned}
\end{equation*}
\end{small}
Note that  the two sets of matrix variables $\{Y\}$ and  $ \{Q_0, Q_{t+1}, T_{1}, \ldots, T_t\}$ in $f_0(\tau;T,Y)$ are separate,  we can do Taylor expansions near  the maximum points respectively, according to   two maximum lemmas  about  $f_0(\tau;T,Y)$ shown in  \cite[Lemma 2.4 and Lemma 2.5]{LZ23}. They  play a  crucial   role in  verifying   local eigenvalue statistics of the deformed GinUEs.  

 By restricting the  integration region to 
\begin{equation}\label{regiondelta}
\Omega_{N,\delta}=\Omega
\cap
A_{N,\delta},
\end{equation}
where
\begin{equation}\label{Omega}
\Omega=\bigg\{ \big(\{T_\alpha\}, Y, Q_0,Q_{t+1}\big)\Big|\sum_{\alpha=1}^t
T_{\alpha}T_{\alpha}^*+Q_0Q_0^*+Q_{t+1}Q_{t+1}^*\leq \mathbb{I}_n \bigg\}
\end{equation}
and
\begin{equation}\label{ANdelta}
A_{N,\delta}=\left\{ \big(\{T_\alpha\}, Y, Q_0,Q_{t+1}\big)\Big|
S(T,Y,Q)
\leq \delta \right\},
\end{equation}
with positive constant $\delta$ and
\begin{equation}\label{SYUT}
\begin{aligned}
S(T,Y,Q)&=\sum_{\alpha=1}^t
\left(
{\rm Tr}\big(
T_{{\rm d},\alpha}-\tau c_{\alpha} |z_0-a_{\alpha}|^{-2}
\mathbb{I}_n
\big)
\big(
T_{{\rm d},\alpha}-\tau c_{\alpha} |z_0-a_{\alpha}|^{-2}\mathbb{I}_n
\big)^*
+{\rm Tr}(T_{{\rm u},\alpha}T_{{\rm u},\alpha}^*)
\right)        \\
&+
{\rm Tr}\big(
YY^*+Q_0Q_0^*+Q_{t+1}Q_{t+1}^*
\big),
\end{aligned}
\end{equation}
we have 
\begin{proposition}\label{RNn delta} With the same notation  and assumptions as in  Proposition \ref{foranlysis}, if $P_{00}(z_0)=1/\tau$, then for any positive constant $\delta$ there exists $\Delta>0$ such that
\begin{equation}\label{RNndelta}
R_N^{(n)}(A_0;z_1,\cdots,z_n)=
D_{N,n}\,\Big(
I_{N,\delta}+O\big(
e^{-\frac{1}{2}N\Delta}
\big)
\Big),
\end{equation}
where
\begin{equation}\label{INdelta}
I_{N,\delta}=
  \int_{\Omega_{N,\delta}} g(T,Y,Q)\exp\Big\{ N\big(f(T,Y,Q)-f_0\big(\tau;\big\{\frac{\tau c_{\alpha}}{f_{\alpha}}\mathbb{I}_n\big\},0,0\big)\big)
 \Big\} {\rm d}V,
\end{equation}
  and 
\begin{equation}\label{DNn}
{\rm d}V= {\rm d}Y
{\rm d}Q_0{\rm d}Q_{t+1} \prod_{\alpha=1}^t
{\rm d}T_{\alpha},\quad D_{N,n}=\widetilde{C}_{N,n,\tau}
e^{nN(|z_0|^2/\tau-1)}\prod_{\alpha=1}^t
(\tau c_{\alpha})^{nNc_{\alpha}}.
\end{equation}
\end{proposition}

\begin{proof}[Proof of Proposition \ref{RNn delta}.] 

\hspace*{\fill}
Using the  two maximum lemmas  about  $f_0(\tau;T,Y)$  in  \cite[Lemma 2.4 and Lemma 2.5]{LZ23},    we see that $f_0(\tau;T,Y,Q)$ in \eqref{f0TY}  attains its maximum  only at  
\begin{small}
\begin{equation}\label{maximumpoint}
Y=0, \ Q_0=0, \ Q_{t+1}=0, \quad T_{\alpha}=\sqrt{\tau c_{\alpha}/f_{\alpha}}\mathbb{I}_n,\quad \alpha=1,\cdots,t. 
\end{equation}
\end{small}
So  there exists $\Delta>0$ such that 
\begin{small}
\begin{equation*}
f_0(\tau;T,Y,Q)-f_0\big(\tau;\big\{\frac{\tau c_{\alpha}}{f_{\alpha}}\mathbb{I}_n\big\},0,0\big)\leq -\Delta
\end{equation*}
\end{small}
in  the domain $\Omega\cap
A_{N,\delta }^{c}$. 

On one hand,  noting that  $$ \sum_{\alpha=1}^t
T_{\alpha}T_{\alpha}^*+Q_0Q_0^*+Q_{t+1}Q_{t+1}^*\leq \mathbb{I}_n$$
in the domain   $\Omega$ defined by  \eqref{Omega}, 
we know  for 
$\hat{Z}$
in a compact subset of $\mathbb{C}^n$ that  
 the partial  sub-leading terms of $f(T,Y,Q)$ in  \eqref{fTY} can be   bounded  by  $O(N^{-\frac{1}{2}})$, that is, 
 \begin{small}
\begin{equation}\label{error term1}
\begin{aligned}
&N^{-\frac{1}{2}}\sum_{\alpha=1}^t 
\left(
\overline{a}_{\alpha}{\rm Tr}\big(
\hat{Z} T_{\alpha}T_{\alpha}^*
\big)+a_{\alpha}{\rm Tr}\big(
\hat{Z}^* T_{\alpha}T_{\alpha}^*
\big)
\right) 
+N^{-\frac{1}{2}}
\left(
\overline{z}_0{\rm Tr}\big(
\hat{Z} Q_0Q_0^*
\big)+z_0{\rm Tr}\big(
\hat{Z}^* Q_0Q_0^*
\big)\right)
\\
&+N^{-\frac{1}{2}}
\left(
{\rm Tr}\big(
\hat{Z} Q_{t+1}A_{t+1}^*Q_{t+1}^*
\big)+{\rm Tr}\big(
\hat{Z}^* Q_{t+1}A_{t+1}Q_{t+1}^*
\big)
\right)
=O(N^{-\frac{1}{2}}).
\end{aligned}
\end{equation}
\end{small}
On the other hand,  apply the singular value decomposition \eqref{singularcorre} for $Y$ and note that  all $f_{\alpha}\not=0$ because of the restriction $P_{00}(z_0)=1$, we obtain 
\begin{small}
\begin{equation*}
\begin{aligned}
&(
\sqrt{\gamma_N}f_{\alpha}\mathbb{I}_n+Y^*Y)^{-1}=O(1),\quad Y(\sqrt{\gamma_N}f_{\alpha}\mathbb{I}_n+Y^*Y)^{-1}=O(1),
\end{aligned}
\end{equation*}
\end{small}
and 
\begin{small}
\begin{equation}\label{error term2}
\log\det\left(
\mathbb{I}_{2n}+\sqrt{\gamma_N} N^{-\frac{1}{2}}
\widehat{A}_{\alpha}
\right)=O(N^{-\frac{1}{2}})
\end{equation}
\end{small}
since $\widehat{A}_{\alpha}=O(1)$. 
Meanwhile, 
\begin{small}
\begin{equation*}
\log\det\left(
\gamma_N f_{\alpha}\mathbb{I}_{n}+
YY^*
\right)=\log\det\left(
f_{\alpha}\mathbb{I}_{n}+
YY^*
\right)+ O\big( N^{-1} \big).
\end{equation*}
\end{small}
Combination of  \eqref{fTY}, \eqref{logAalpha decompose}, \eqref{f0TY}, \eqref{error term1} and \eqref{error term2}  give rise to 
\begin{small}
\begin{equation*}
N\Big(
 f(T,Y,Q)-f_0\big(\tau;\big\{\frac{\tau c_{\alpha}}{f_{\alpha}}\mathbb{I}_n\big\},0,0\big)
\Big)=
N\Big(
f_0(\tau;T,Y,Q)-f_0\big(\tau;\big\{\frac{\tau c_{\alpha}}{f_{\alpha}}\mathbb{I}_n\big\},0,0\big)
\Big)+
\frac{n}{\tau}{\rm Tr}(YY^*)+O\big(
\sqrt{N}
\big).
\end{equation*}
\end{small}

For $g(T,Y,Q)$ in \eqref{gYUT},  all factors can be controlled by a  polynomial function, so by taking a sufficiently large $N_0>\sum_{\alpha=1}^t (n-R_{\alpha,N})/c_{\alpha}$ we know 
\begin{small}
\begin{equation*}
\int_{A_{N,\delta}^{c}\cap \Omega}
| g(T,Y,Q) |\exp\left\{ N_0\Big(
 f_0(\tau;T,Y,Q)-f_0\big(\tau;\big\{\frac{\tau c_{\alpha}}{f_{\alpha}}\mathbb{I}_n\big\},0,0\big)
\Big)+\frac{n}{\tau}{\rm Tr}(YY^*) \right\}
{\rm d}V
=O(1).
\end{equation*}
\end{small}
Hence,
\begin{small}
\begin{equation*}
\begin{aligned}
&\left|
\int_{A_{N,\delta}^{\complement}\cap \Omega}
g(T,Y,Q) \exp\left\{ N\Big(
 f(T,Y,Q)-f_0\big(\tau;\big\{\frac{\tau c_{\alpha}}{f_{\alpha}}\mathbb{I}_n\big\},0,0\big)
\Big) \right\}
{\rm d}V
\right|   \\
&\leq e^{-(N-N_0)\Delta+O(\sqrt{N})}
 O(1)    =O\big(
e^{-\frac{1}{2}N\Delta}
\big).
\end{aligned}
\end{equation*}
\end{small}

This  thus completes the proof.
\end{proof}

The remaining task is to analyze the matrix integral $I_{N,\delta}$ with sufficiently small positive $\delta$, which 
follows the argument of   Laplace method. To be precise, the proof  of Theorem   \ref{2-complex-correlation} will be divided into three steps: {\bf Step 1} deals with Taylor expansions  of $f(T,Y,Q)$ and  of $g(T,Y,Q)$. According to  the integration region in \eqref{integrationregionprop},   the most relevant variables are triangular matrices $\big\{ T_{\alpha} \big\}$, equivalently, diagonal matrices $\big\{ T_{{\rm d},\alpha} \big\}$ and strictly upper triangular matrices $\big\{ T_{{\rm u},\alpha} \big\}$ as in \eqref{Talpha 2}. With the coming of the change of variables
\begin{equation}\label{prechange1}
T_{{\rm d},1},T_{{\rm d},2},\cdots,T_{{\rm d},t}\longmapsto
\widetilde{T}_{{\rm d},1},\widetilde{T}_{{\rm d},2},\cdots,\widetilde{T}_{{\rm d},t}\longmapsto \Gamma_1,\widetilde{T}_{{\rm d},2},\cdots,\widetilde{T}_{{\rm d},t}
\end{equation}
and 
\begin{equation}\label{prechange2}
T_{{\rm u},1},T_{{\rm u},2},\cdots,T_{{\rm u},t}\longmapsto
G_1,T_{{\rm u},2},\cdots,T_{{\rm u},t};
\end{equation}
{\bf Step 2} is a standard procedure in Laplace method to remove the error terms resulted from the expansions in first two steps; In {\bf Step 3}, after dealing with several matrix integrals, we can complete the proof.

\subsection{Proof of Theorem   \ref{2-complex-correlation}: $z_0\neq 0$} \label{case1}

The proof of Theorem  is  divided into two cases, one  deals with the case when $z_0\not=0$ while  the other   deals with the case when $z_0=0$. With Section \ref{sectnotation} in mind, we still need some extra notations. For a family of matrices $\{ T_{{\rm u},\alpha}|\alpha=1,\cdots,t\}$, define 
$\|T_{\rm u}\|:=\sum_{\alpha=1}^t \|T_{{\rm u},\alpha}\|,\quad    
     \| T_{\rm u}\|_{2} :=\sum_{\alpha=2}^t
 \| T_{{\rm u},\alpha}\|.$
 Similarly,  for a family of matrices $\{ \widetilde{T}_{{\rm d},\alpha}|\alpha=1,\cdots,t\}$, define
$
  \|\widetilde{T}_{{\rm d}}\|:=\sum_{\alpha=1}^t
 \| \widetilde{T}_{{\rm d},\alpha}\|, \ 
 \| 
  \widetilde{T}_{{\rm d}}\|_{2} :=\sum_{\alpha=2}^t
 \| 
  \widetilde{T}_{{\rm d},\alpha}\|.
$
 For a complex matrix $M$, denote $M^{({\rm d})}$ the diagonal part of M, and $M^{({\rm off},{\rm u})}$ the strictly upper triangular part of M;  in particular,  for $M_{\alpha}$ with index $\alpha$,   the corresponding quantity  is denoted by   $M^{(\alpha,{\rm d})}$ and $M^{(\alpha,{\rm off},{\rm u})}$.
 
 We need the following two propositions concerning the Taylor expansions, whose proofs are postponed   until  next section.  
\begin{proposition} \label{fTaylor}

With $f(T,Y,Q)$  in \eqref{fTY}, we have 
\begin{equation}\label{fTYexpanNon0}
f(T,Y,Q)-f_0\big(\tau;\big\{\frac{\tau c_{\alpha}}{f_{\alpha}}\mathbb{I}_n\big\},0,0\big)
=F_0+F_1+O(F_2),
\end{equation}
where
\begin{equation*}
\begin{aligned}
F_0&=\frac{N^{-\frac{1}{2}}}{\tau}\left(
\overline{z_0}
{\rm Tr}\big(
\hat{Z}
\big)+
z_0
{\rm Tr}\big(
\hat{Z}^*
\big)
\right) +\frac{n^2}{N}       \\
&-\frac{N^{-1}}{2}
\sum_{\alpha=1}^t\frac{c_{\alpha}}{f_{\alpha}^2}
\left(
\overline{z_0-a_{\alpha}}^2{\rm Tr}\big(
\hat{Z}^2
\big)+
(z_0-a_{\alpha})^2{\rm Tr}\big(
\hat{Z}^*
\big)^2
\right)  ,
\end{aligned}
\end{equation*}
\begin{small}
\begin{equation}\label{f1Non0}
\begin{aligned}
&F_1=\frac{|z_0|^2}{\tau}{\rm Tr}(\Gamma_1)-
\frac{1}{\tau}{\rm Tr}Q_{t+1}
\big(
z_0\mathbb{I}_{r_{t+1}}-A_{t+1}
\big)^*
\big(
z_0\mathbb{I}_{r_{t+1}}-A_{t+1}
\big)
Q_{t+1}^*
  -\sum_{\alpha=2}^t\frac{f_{\alpha}^2}{2\tau^2 c_{\alpha}}{\rm Tr}\big(\widetilde{T}_{{\rm d},\alpha}^2  \big)         \\
&-\frac{f_{1}^2}{2\tau^2 c_{1}}{\rm Tr}
\left(
 (Q_0Q_0^*)^{({\rm d})}+\sum_{\alpha=2}^t
\widetilde{T}_{{\rm d},\alpha}
 \right)^2
-\frac{N^{-\frac{1}{2}}}{\tau}{\rm Tr}\Big(
(\overline{a}_{1}\hat{Z}+a_{1}\hat{Z}^*)
\big(
(Q_0Q_0^*)^{({\rm d})}+\sum_{\alpha=2}^t \widetilde{T}_{{\rm d},\alpha}
\big)
\Big)     \\
&+\frac{N^{-\frac{1}{2}}}{\tau}\sum_{\alpha=2}^t{\rm Tr}
\big((\overline{a}_{\alpha}\hat{Z}+a_{\alpha}\hat{Z}^*)
\widetilde{T}_{{\rm d},\alpha}\big)+
\frac{N^{-\frac{1}{2}}}{\tau}
\left(
\overline{z}_0{\rm Tr}\big(
\hat{Z} Q_0Q_0^*
\big)+z_0{\rm Tr}\big(
\hat{Z}^* Q_0Q_0^*
\big) 
  \right)     \\
&-\frac{f_1^2}{\tau^2 c_1}
 {\rm Tr}\left(
 (Q_0Q_0^*)^{({\rm off},{\rm u})}+\sum_{\alpha=2}^t
 \sqrt{\frac{\tau c_{\alpha}}{f_{\alpha}}}T_{{\rm u},\alpha}
 \right)
 \left(
 (Q_0Q_0^*)^{({\rm off},{\rm u})}+\sum_{\alpha=2}^t
 \sqrt{\frac{\tau c_{\alpha}}{f_{\alpha}}}T_{{\rm u},\alpha}
 \right)^*  \\
 &-\sum_{\alpha=2}^t
\frac{f_{\alpha}}{\tau}{\rm Tr}\big(
T_{{\rm u},\alpha}T_{{\rm u},\alpha}^*
\big) +\frac{1}{\tau}\sum_{1\leq i < j\leq n}
 \left|
(z_0-a_1) (Q_0Q_0^*)^{({\rm off},{\rm u})}_{i,j}
+\sum_{\alpha=2}^t
\sqrt{\frac{\tau c_{\alpha}}{f_{\alpha}}}
(a_{\alpha}-a_1)
t_{i,j}^{(\alpha)}
 \right|^2   \\
 &-\frac{P_1}{2} {\rm Tr}(YY^*)^2 
+N^{-\frac{1}{2}}\left(
\overline{P_0}{\rm Tr}\big(
Y^*Y\hat{Z}
\big)+P_0{\rm Tr}\big(
YY^*\hat{Z}^*
\big)
\right),
\end{aligned}
\end{equation}
\end{small}
and
\begin{small}
\begin{equation*}
\begin{aligned}
F_2&=N^{-\frac{3}{2}}+\| Y \|^6+N^{-1}\| Y \|^2+N^{-\frac{1}{2}}\| Y \|^4
                 \\
&+ \Omega_{{\rm error}}^3+
N^{-\frac{1}{2}}\big(
\|\Gamma_1\|+\Omega_{{\rm error}}^2
\big)+\big(
\|\Gamma_1\|+\|G_1\|
\big)\Omega_{{\rm error}}.
\end{aligned}
\end{equation*}
\end{small}
Here  the  changes of variables  in    \eqref{Tdalpha1} \eqref{matrix transformations1} and  \eqref{matrix transformations2} have been used,   and  $\Omega_{{\rm error}}$ is defined in  \eqref{Omegaerror}, in Section \ref{propproofs}.   \end{proposition}

\begin{proposition} \label{gTaylor}

With $g(T,Y,Q)$  in \eqref{gYUT}, we have 
\begin{equation*} 
\begin{aligned}
g(T,Y,Q) &=
 \Big(\big(|z_0|^{2n}+O\big(
N^{-\frac{1}{2}}+\| Y \|
\big)\big)\det\big(
-\Gamma_1-G_1-G_1^*
\big)\Big)^{R_0-n}\\
& \times  \prod_{\alpha=1}^t\prod_{j=1}^n
\Big(
   \frac{\tau c_{\alpha}}{f_{\alpha}}
+O\big(\Omega_{\rm error}
\big)\Big)^{R_{\alpha,N}-j}. \prod_{\alpha=1}^t\Big(f_{\alpha}^n
+O\big(
N^{-\frac{1}{2}}+\| Y \|
\big)\Big)^{R_{\alpha,N}-n}\\
&\times 
\big| \det\big( z_0\mathbb{I}_{r_{t+1}}-A_{t+1} \big) \big|^{2n}
\Big(
\big(\det(YY^*)\big)^{r_0}
 +O\Big(
\sum_{\alpha,\beta}\| Y \|^{\alpha}
\Omega_{\rm error}^{\beta}
\Big)
\Big)
\\&\times\Big(
\tau^{2n^2}|z_0|^{2n^2}|P_0|^{2n^2}
\prod_{\alpha=1}^t
f_{\alpha}^{n^2}
 +O\big(
\Omega_{\rm error}+N^{-\frac{1}{2}}+\| Y \|
\big)
\Big)   
\end{aligned}
\end{equation*} 
where the finite non-negative integer pair $(\alpha,\beta)$ satisfies 
 \eqref{M error condition} and 
   $\Omega_{{\rm error}}$ is defined in  \eqref{Omegaerror}, in Section \ref{propproofs}.   \end{proposition}

\begin{proof}[Proof of Theorem   \ref{2-complex-correlation}: $z_0\neq 0$]
We proceed in three steps. 

{\bf Step 1: Taylor expansions of $f(T,Y,Q)$} and  $g(T,Y,Q)$.

For  $z_0\not=0$,  by Proposition \ref{RNn delta}  we can  restrict all the matrix variables in the region $\Omega_{N,\delta}$ defined in \eqref{regiondelta} for sufficiently small $\delta>0$. Hence, we can do Taylor expansions of the relevant matrix variables   as stated  in  Proposition \ref{fTaylor}  and Proposition \ref{gTaylor}.

{\bf Step 2: Removal of error terms}.

 
Recalling \eqref{INdelta}, \eqref{Jacobiandet} and \eqref{fTYexpanNon0}, we rewrite
 \begin{equation}\label{INdeltaexpan}
e^{-NF_0}I_{N,\delta}=J_{1,N}+J_{2,N},
\end{equation}
 where 
 \begin{equation*}
J_{1,N}=\Big(
\frac{f_{1}}{\tau c_{1}}
\Big)^{\frac{n(n-1)}{2}}\int_{\hat{\Omega}_{N,\delta}}
g(T,Y,Q)\big(1+O\big(
\Omega_{\rm error}
\big)\big)e^{NF_1}
{\rm d}\widetilde{V},
\end{equation*}
and
 \begin{equation*}
J_{2,N}=\Big(
\frac{f_{1}}{\tau c_{1}}
\Big)^{\frac{n(n-1)}{2}}\int_{\hat{\Omega}_{N,\delta}}
g(T,Y,Q)e^{NF_1}\big(
e^{O(NF_2)}-1
\big)\big(1+O\big(
\Omega_{\rm error}
\big)\big)
{\rm d}\widetilde{V},
\end{equation*}
with
\begin{equation*}
{\rm d}\widetilde{V}={\rm d}Y
{\rm d}Q_0{\rm d}Q_{t+1}{\rm d}\Gamma_1
{\rm d}G_1\prod_{\alpha=2}^t {\rm d}\widetilde{T}_{{\rm d},\alpha}
{\rm d}T_{{\rm u},\alpha}
\end{equation*}
and
\begin{equation*}
\hat{\Omega}_{N,\delta}=A_{N,\delta}\cap \big\{
\Gamma_1+G_1+G_1^*\leq 0
\big\},
\end{equation*}
 cf.  \eqref{ANdelta} for definition of $A_{N,\delta}$.
 
Under  the restriction condition  of $\Gamma_1+G_1+G_1^* \leq 0$,  every   principal minor  of order 2 is non-positive definite, so all diagonal entries of $\Gamma_1$  are zero or negative and  
\begin{equation}\label{g2hatU1 estimation}
{\rm Tr}\big(
G_1G_1^*
\big)=\sum_{1\leq i<j\leq n}\big| (G_1)_{i,j} \big|^2
\leq  \sum_{1\leq i<j\leq n} (\Gamma_1)_{i,i} (\Gamma_1)_{j,j}
\leq  
\big(  {\rm Tr}
(\Gamma_1)
\big)^2.
\end{equation}
Noticing the absence of $ G_1$ in $F_1$ given  in \eqref{f1Non0}, for convergence we need to control  $G_1$ by $
\Gamma_1$  in $F_2$,  just as shown in  \eqref{g2hatU1 estimation}. Since  $\| G_1 \|= O(\| \Gamma_1 \|)$ from   \eqref{g2hatU1 estimation}, also from \eqref{regiondelta} and \eqref{ANdelta} for sufficiently large $N$ and small $\delta$, there exists some $C>0$ such that 
  \begin{small}
\begin{equation*}
\begin{aligned}
&\frac{1}{C} NF_2 \leq N^{-\frac{1}{2}}+\sqrt{\delta}+
\sqrt{N\delta} \big(
\| Y \|^2+\|\widetilde{T}_{{\rm d}} \|_2
+\| Q_0 \|^2+\| Q_{t+1} \|+\|T_{{\rm u}} \|_2
\big)            \\
&+N\sqrt{\delta}\big(
\| Q_0 \|^4
+\|\widetilde{T}_{{\rm d}} \|_2^2+\|T_{{\rm u}} \|_2^2
+\| Y \|^4+\| Q_{t+1} \|^2+\| \Gamma_1 \|
\big).
\end{aligned}
\end{equation*}
\end{small}

 Using the inequality \begin{equation}\label{J2N inequ}
\big|
e^{O(NF_2)}-1
\big|\leq O(N|F_2|)e^{O(N|F_2|)},
\end{equation}
 after change of  variables
\begin{equation}\label{change scale}
\begin{aligned}
&(G_1,\Gamma_1)\rightarrow N^{-1}(G_1,\Gamma_1),\quad
(Q_0,Y)\rightarrow N^{-\frac{1}{4}}(Q_0,Y),\quad
Q_{t+1} \rightarrow N^{-\frac{1}{2}}Q_{t+1},
\\
&\big(
\widetilde{T}_{{\rm d},\alpha},T_{{\rm u},\alpha}
\big)\rightarrow N^{-\frac{1}{2}} 
\big(
\widetilde{T}_{{\rm d},\alpha},T_{{\rm u},\alpha}
\big)
\quad \alpha=2,\cdots,t,
\end{aligned}
\end{equation}
the term $O(NF_2)$ in \eqref{J2N inequ} has an upper bound  by   $N^{-\frac{1}{4}}P(\tilde{V})$ for some polynomial of variables.  
since $F_1$ can control $F_2$  for sufficiently  small $\delta$, by the argument of  Laplace method and the dominant convergence theorem we know that 
$J_{2,N}$ is typically of order  $N^{-\frac{1}{4}}$ compared with $J_{1,N}$, that is,  
 \begin{equation}\label{J2Nestimation}
J_{2,N}=O\big(
N^{-\frac{1}{4}}
\big)J_{1,N}.
\end{equation}
For $J_{1,N}$,   
take  a large  $M_0>0$ such that
\begin{small}
\begin{equation*}
\begin{aligned}
A_{N,\delta}^{\complement}\subseteq  
&   \bigcup_{\alpha=2}^t
\Big\{
{\rm Tr}\big(
T_{{\rm u},\alpha}T_{{\rm u},\alpha}^*
\big)>\frac{\delta}{M_0}
\Big\}
\bigcup
\bigcup_{\alpha=2}^t
\Big\{
{\rm Tr}\big(
\widetilde{T}_{{\rm d},\alpha}\widetilde{T}_{{\rm d},\alpha}^*
\big)>\frac{\delta}{M_0}
\Big\} \bigcup \Big\{{\rm Tr}\big(
\Gamma_{1}\Gamma_{1}^*
\big)>\frac{\delta}{M_0}\Big\}
        \\
&\bigcup \Big\{{\rm Tr}\big(
Q_0Q_0^*
\big)>\frac{\delta}{M_0}\Big\} 
 \bigcup
\Big\{
{\rm Tr}\big(
YY^*
\big)>\frac{\delta}{M_0}
\Big\}
\bigcup
\Big\{
{\rm Tr}\big(
Q_{t+1}Q_{t+1}^*
\big)>\frac{\delta}{M_0}
\Big\}.
\end{aligned}
\end{equation*}
\end{small}
Here we have used \eqref{g2hatU1 estimation} to drop out the domain   
$ \{{\rm Tr}\big(
G_1G_1^*
\big)>\delta/M_0\}$.   For each piece of  domain,  only  keep  the restricted matrix variable and let  the others free,  it's easy to prove that the corresponding  matrix integral  is exponentially  small, that is,  $O\big( e^{-\delta_1 N} \big)
$
for some $\delta_1>0$.  

So we can extend the integration  region from $\hat{\Omega}_{N,\delta}$ to  $\big\{ \Gamma_1+G_1+G_1^* \leq 0 \big\}$. From \eqref{ralpha N} and \eqref{calpha sum}, we have $\sum_{\alpha=1}^t R_{\alpha,N}=-r_0-r_{t+1}-R_0$, and by  the change of  variables    
\eqref{change scale} we have
\begin{equation}\label{J1Nform}
\begin{aligned}
J_{1,N}&=N^{-\frac{n^2 t}{2}-n(R_0+r_0+r_{t+1})}
\Big(
\frac{f_1}{\tau c_1}
\Big)^{\frac{n(n-1)}{2}}
|z_0|^{2nR_0}|P_0|^{2n^2}
\tau^{2n^2-n(r_0+r_{t+1}+R_0)-\frac{n(n+1)}{2}t }
           \\
&\times \Big|
\det \big(
z_0\mathbb{I}_{r_{t+1}}-A_{t+1}
\big)
\Big|^{2n}
\prod_{\alpha=1}^t c_{\alpha}^{nR_{\alpha,N}-\frac{n(n+1)}{2}}
f_{\alpha}^{\frac{n(n+1)}{2}}\Big(
I_0+O\big(
N^{-\frac{1}{4}}
\big)
\Big),
\end{aligned}
\end{equation}
where with $ \widehat{H}=\Gamma_1+G_1+G_1^*$,
\begin{equation}\label{I0}
I_0=\int_{\widehat{H}\leq 0}\big(
\det(-\widehat{H})
\big)^{R_0-n}
\big(
\det(YY^*)
\big)^{r_0}
e^F
{\rm d}\widetilde{V},
\end{equation}
\begin{small}
\begin{equation*}
\begin{aligned}
&F=\frac{|z_0|^2}{\tau}{\rm Tr}(\Gamma_1)-\frac{1}{\tau}{\rm Tr}Q_{t+1}
\big(
z_0\mathbb{I}_{r_{t+1}}-A_{t+1}
\big)^*
\big(
z_0\mathbb{I}_{r_{t+1}}-A_{t+1}
\big)
Q_{t+1}^*  
-\sum_{\alpha=2}^t\frac{f_{\alpha}^2}{2\tau^2 c_{\alpha}}
{\rm Tr}\big( \widetilde{T}_{{\rm d},\alpha}^2 \big)     \\
&-\frac{f_{1}^2}{2\tau^2 c_{1}}{\rm Tr}
\Big(
 (Q_0Q_0^*)^{({\rm d})}+\sum_{\alpha=2}^t
\widetilde{T}_{{\rm d},\alpha}
 \Big)^2 
 -\frac{1}{\tau}{\rm Tr}\Big(
\big(
\overline{a}_1\hat{Z}+a_1\hat{Z}^*
\big)\big(
(Q_0Q_0^*)^{({\rm d})}+\sum_{\alpha=2}^t \widetilde{T}_{{\rm d},\alpha}
\big)
\Big)       \\
&+\frac{1}{\tau}\sum_{\alpha=2}^t {\rm Tr}
\Big(\big(
\overline{a}_{\alpha}\hat{Z}+a_{\alpha}\hat{Z}^*
\big)\widetilde{T}_{{\rm d},\alpha}\Big)+
\frac{1}{\tau}\Big(
\overline{z}_0{\rm Tr}\big(
\hat{Z} Q_0Q_0^*
\big)+z_0{\rm Tr}\big(
\hat{Z}^* Q_0Q_0^*
\big) 
  \Big)     \\
&-\frac{f_1^2}{\tau^2 c_1}
 {\rm Tr}\Big(
 (Q_0Q_0^*)^{({\rm off},{\rm u})}+\sum_{\alpha=2}^t
 \sqrt{\frac{\tau c_{\alpha}}{f_{\alpha}}}T_{{\rm u},\alpha}
 \Big)
 \Big(
(Q_0Q_0^*)^{({\rm off},{\rm u})}+\sum_{\alpha=2}^t
 \sqrt{\frac{\tau c_{\alpha}}{f_{\alpha}}}T_{{\rm u},\alpha}
 \Big)^*  \\
 &-\sum_{\alpha=2}^t \frac{1}{\tau}
f_{\alpha}{\rm Tr}\big(
T_{{\rm u},\alpha}T_{{\rm u},\alpha}^*
\big) + \frac{1}{\tau}\sum_{1\leq i < j\leq n}
 \left|
(z_0-a_1) (Q_0Q_0^*)^{({\rm off},{\rm u})}_{i,j}
+\sum_{\alpha=2}^t
\sqrt{\frac{\tau c_{\alpha}}{f_{\alpha}}}
(a_{\alpha}-a_1)
t_{i,j}^{(\alpha)}
 \right|^2 \\
 &-\frac{P_1}{2} {\rm Tr}(YY^*)^2 
+
\overline{P_0}{\rm Tr}\big(
Y^*Y\hat{Z}
\big)+P_0{\rm Tr}\big(
YY^*\hat{Z}^*
\big)
\end{aligned}
\end{equation*}
\end{small}

{\bf Step 3: Matrix integrals and final proof.} 

To simplify $I_0$ further, we first integrate out $Q_{t+1},T_{{\rm u},2},\cdots,T_{{\rm u},t}$,
\begin{equation*}
\int
e^{-\frac{1}{\tau}{\rm Tr}Q_{t+1}
\big(
z_0\mathbb{I}_{r_{t+1}}-A_{t+1}
\big)^*
\big(
z_0\mathbb{I}_{r_{t+1}}-A_{t+1}
\big)
Q_{t+1}^* }{\rm d}Q_{t+1}        
=
 \Big|
\det \big(
z_0\mathbb{I}_{r_{t+1}}-A_{t+1}
\big)
\Big|^{-2n}
(\pi\tau)^{n r_{t+1}}.
\end{equation*}
From Proposition \ref{matrixintegral2}, we have
\begin{small}
\begin{equation*}
\begin{aligned}
&\int
\exp\Big\{
-\frac{f_1^2}{\tau^2 c_1}
 {\rm Tr}\Big(
 (Q_0Q_0^*)^{({\rm off},{\rm u})}+\sum_{\alpha=2}^t
 \sqrt{\frac{\tau c_{\alpha}}{f_{\alpha}}}T_{{\rm u},\alpha}
 \Big)
 \Big(
(Q_0Q_0^*)^{({\rm off},{\rm u})}+\sum_{\alpha=2}^t
 \sqrt{\frac{\tau c_{\alpha}}{f_{\alpha}}}T_{{\rm u},\alpha}
 \Big)^*
 \\
&-\sum_{\alpha=2}^t
\frac{1}{\tau} f_{\alpha}{\rm Tr}\big(
T_{{\rm u},\alpha}T_{{\rm u},\alpha}^*
\big) +\frac{1}{\tau}\sum_{1\leq i < j\leq n}
 \left|
(z_0-a_1) (Q_0Q_0^*)^{({\rm off},{\rm u})}_{i,j}
+\sum_{\alpha=2}^t
\sqrt{\frac{\tau c_{\alpha}}{f_{\alpha}}}
(a_{\alpha}-a_1)
t_{i,j}^{(\alpha)}
 \right|^2
 \Big\}
 \\
 &\times
 \prod_{\alpha=2}^t {\rm d}T_{{\rm u},\alpha}    
 =
\tau^{\frac{n(n-1)}{2}(t-2)} 
 \Big(
\frac{f_1}{c_1}
\big( \prod_{\alpha=1}^t f_{\alpha} \big)
|P_0|^2
 \Big)^{-\frac{n(n-1)}{2}}
 \pi^{\frac{n(n-1)}{2}(t-1)},
\end{aligned}
\end{equation*}
\end{small}
which is independent of $Q_0$. 

Secondly, to integrate out $G_1$ and $\Gamma_1$, by Proposition \ref{matrixintegral1}, we have
\begin{small}
\begin{equation*}
\int_{\widehat{H}\leq 0}\big(
\det(-\widehat{H})
\big)^{R_0-n}{\rm d}G_1
=(-1)^{n(R_0-n)}\pi^{\frac{n(n-1)}{2}}
\frac{\prod_{i=1}^n(R_0-i)!\big( \det(\Gamma_1) \big)^{R_0-1}}
{\big((R_0-1)!\big)^n},
\end{equation*}
\end{small}
from which
\begin{small}
\begin{equation}\label{Gamma integral}
(-1)^{n(R_0-n)}\int_{\Gamma_1\leq 0}
\big( \det(\Gamma_1) \big)^{R_0-1}
e^{\frac{1}{\tau}|z_0|^2{\rm Tr}(\Gamma_1)}{\rm d}\Gamma_1
=\tau^{nR_0}|z_0|^{-2nR_0}\big((R_0-1)!\big)^n.
\end{equation}
\end{small}
Thirdly, we need to integrate out $\widetilde{T}_{{\rm d},\alpha}$ for $\alpha=2,\cdots,t$. For $i=1,\cdots,n$, let
\begin{equation*}
\vec{b}_i=\big(
b_i^{(2)},b_i^{(3)},\cdots,b_i^{(t)}
\big),\quad
b_i^{(\alpha)}=
\Big(
\overline{a_{\alpha}-a_1}\hat{Z}
+(a_{\alpha}-a_1)\hat{Z}^*
-\frac{f_1^2}{\tau c_1}
 Q_0Q_0^*
 \Big)_{i,i},
\end{equation*}
and introduce a square matrix
\begin{equation*} 
\Sigma=[\Sigma_{\alpha,\beta}]_{\alpha,\beta=2}^t,  \quad \Sigma_{\alpha,\beta}=\frac{f_{\alpha}^2}{2c_{\alpha}}
\delta_{\alpha,\beta}+\frac{f_{1}^2}{2c_{1}},
\end{equation*} 
Noting
\begin{equation}\label{Ainverse critical non0}
\det(\Sigma)=2P_1\prod_{\alpha=1}^t \frac{f_{\alpha}^2}{2c_{\alpha}}
,\quad
(\Sigma^{-1})_{\alpha,\beta}=
\frac{2c_{\alpha}\delta_{\alpha,\beta}}{f_{\alpha}^2}
-\frac{2c_{\alpha}c_{\beta}}{f_{\alpha}^2f_{\beta}^2P_1},
\quad
\alpha,\beta=2,\cdots,t,
\end{equation} 
where $P_1$ is given in \eqref{parameter2}. Calculate the gaussian integrals and we obtain
\begin{equation}\label{Ualpha integral}
\begin{aligned}
&\int \exp
\Big\{
-\frac{f_{1}^2}{2\tau^2 c_{1}}{\rm Tr}
\big(
 (Q_0Q_0^*)^{({\rm d})}+\sum_{\alpha=2}^t
 \widetilde{T}_{{\rm d},\alpha}
 \big)^2
  -\sum_{\alpha=2}^t\frac{f_{\alpha}^2}{2\tau^2 c_{\alpha}}
  {\rm Tr} \big( \widetilde{T}_{{\rm d},\alpha}^2  \big)
  \\&     
+\frac{1}{\tau}\sum_{\alpha=2}^t{\rm Tr}\Big(
\big(
\overline{a_{\alpha}-a_1}\hat{Z}+
(a_{\alpha}-a_1)\hat{Z}^*
\big)
\widetilde{T}_{{\rm d},\alpha}
\Big)\Big\}
\prod_{\alpha=2}^t{\rm d}\widetilde{T}_{{\rm d},\alpha} \\
&=
\Big(
  2P_1\prod_{\alpha=1}^t \frac{f_{\alpha}^2}{2c_{\alpha}} 
  \Big)
 ^{-\frac{n}{2}}
 \pi^{\frac{t-1}{2}n}
 \tau^{(t-1)n}
e^{-\frac{f_{1}^2}{2\tau^2 c_{1}}{\rm Tr}\big(
 (Q_0Q_0^*)^{({\rm d})}
 \big)^2
 +\frac{1}{4}\sum_{i=1}^n 
 \vec{b}_i \Sigma^{-1}\vec{b}_i^t}
\end{aligned}
\end{equation}
Obviously,
\begin{small}
\begin{equation*}
\begin{aligned}
\sum_{i=1}^n\sum_{\alpha=2}^t
 \frac{c_{\alpha}\big(
 b_i^{(\alpha)}
\big)^2 }
 {f_{\alpha}^2}&=
\sum_{\alpha=2}^t
 \frac{c_{\alpha} }
 {f_{\alpha}^2}{\rm Tr}\Big(
 \big(\overline{a}_{\alpha}\hat{Z}+
 a_{\alpha}\hat{Z}^*\big)^2-2\big(\overline{a}_{\alpha}\hat{Z}+
 a_{\alpha}\hat{Z}^*\big)  \\
 &
 \times\big( \overline{a}_{1}\hat{Z}+
 a_{1}\hat{Z}^*+\frac{f_1^2}{\tau c_1}(Q_0Q_0^*)^{({\rm d})} \big)
 +\big( \overline{a}_{1}\hat{Z}+
 a_{1}\hat{Z}^*+\frac{f_1^2}{\tau c_1}(Q_0Q_0^*)^{({\rm d})} \big)^2
 \Big), 
\end{aligned}
\end{equation*}
\end{small}
\begin{small}
\begin{equation*}
\begin{aligned}
\sum_{i=1}^n
\Big(\sum_{\alpha=2}^t
 \frac{c_{\alpha}
 b_i^{(\alpha)} }
 {f_{\alpha}^2}\Big)^2=
{\rm Tr}\Big(
\sum_{\alpha=1}^t
 \frac{c_{\alpha}}{f_{\alpha}^2}\big(\overline{a}_{\alpha}\hat{Z}+
 a_{\alpha}\hat{Z}^*\big) -P_1\big( \overline{a}_{1}\hat{Z}+
 a_{1}\hat{Z}^*+\frac{f_1^2}{\tau c_1}(Q_0Q_0^*)^{({\rm d})} \big)
 +\frac{1}{\tau}(Q_0Q_0^*)^{({\rm d})}\Big)^2.
 \end{aligned}
\end{equation*}
\end{small}
Use \eqref{Ainverse critical non0}, elementary calculation gives us
\begin{equation*}
\frac{1}{4}
 \vec{b}_i \Sigma^{-1}\vec{b}_i^t
 =
\frac{1}{2}\sum_{\alpha=2}^t
 \frac{c_{\alpha}\big(
 b_i^{(\alpha)}
\big)^2 }
 {f_{\alpha}^2}-\frac{1}{2P_1}
 \Big(\sum_{\alpha=2}^t
 \frac{c_{\alpha}
 b_i^{(\alpha)} }
 {f_{\alpha}^2}\Big)^2,
\end{equation*} 
and
\begin{small}
\begin{equation}\label{diagonal simplfy}
\begin{aligned}
-\frac{f_{1}^2}{2\tau^2 c_{1}}{\rm Tr}
&\big(
 (Q_0Q_0^*)^{({\rm d})}
 \big)^2
 +\frac{1}{\tau}{\rm Tr}\Big(
 \big(
 \overline{z_0-a_1} \hat{Z}+(z_0-a_1)\hat{Z}^*
 \big)(Q_0Q_0^*)^{({\rm d})}
 \Big)
 +\frac{1}{4}\sum_{i=1}^n 
 \vec{b}_i \Sigma^{-1}\vec{b}_i^t     \\
 =-\frac{1}{2P_1}
 &{\rm Tr}
 \Big(
\sum_{\alpha=1}^t   \frac{c_{\alpha}}{f_{\alpha}^2}
\big(\overline{a}_{\alpha}\hat{Z}+
 a_{\alpha}\hat{Z}^*\big)
 \Big)^2
 +\frac{1}{2}\sum_{\alpha=1}^t
 \frac{c_{\alpha}}
 {f_{\alpha}^2}{\rm Tr}\big(\overline{a}_{\alpha}\hat{Z}+
 a_{\alpha}\hat{Z}^*\big)^2
 \\&
 -\sum_{i=1}^n\Big(
 \frac{1}{2P_1\tau^2}\big( (Q_0Q_0^*)^{({\rm d})}_{i,i} \big)^2
 -\frac{1}{P_1\tau}(Q_0Q_0^*)^{({\rm d})}_{i,i}
 \big(
 \overline{P}_0\hat{z}_i+P_0
 \overline{\hat{z}_i}
 \big)
 \Big),
\end{aligned}
\end{equation}
\end{small}

Let $Q_0=[q_{i,j}^{(0)}]$, and apply the spherical coordinates to $(Q_0Q_0^*)^{({\rm d})}_{i,i}$:
\begin{equation*}
\big(
q_{i,1}^{(0)},\cdots,q_{i,r_0}^{(0)}
\big)=\sqrt{q_i}\vec{u}_i,
\quad
\| \vec{u}_i \|=1,
\end{equation*}
with Jacobian
\begin{equation*}
\prod_{j=1}^{r_0}{\rm d}q_{i,j}^{(0)}=
\frac{\pi^{r_0}}{\Gamma(r_0)}q_i^{r_0-1}{\rm d}q_i{\rm d}\vec{u}_i,
\end{equation*}
where ${\rm d}\vec{u}_i$ is the Haar measure on the sphere $S^{r_0}.$ Then
\begin{small}
\begin{equation}\label{erfc integral}
\begin{aligned}
&\prod_{i=1}^n\int 
e^{-\frac{q_i^2}{2\tau^2 P_1}-\frac{q_i}{P_1\tau}\big(
\overline{P}_0\hat{z}_i+P_0\overline{\hat{z}_i}
\big)}\prod_{j=1}^{r_0}
{\rm d}q_{i,j}^{(0)}     
=\frac{(\pi\tau)^{n r_0}}{\big( \Gamma(r_0) \big)^n}
\prod_{i=1}^n\int_0^{+\infty}x^{r_0-1}
e^{-\frac{x^2}{2P_1}-\frac{x}{P_1}\big(
\overline{P}_0\hat{z}_i+P_0\overline{\hat{z}_i}
\big)}{\rm d}x.
\end{aligned}
\end{equation}
\end{small}

Finally, with \eqref{norm-1} and \eqref{DNn} in mind, by the Stirling's formula we can obtain
\begin{small}
\begin{equation}\label{eNf0 DNn}
\begin{aligned}
&e^{NF_0}D_{N,n}=\frac
{N^{\frac{n^2t}{2}+n(R_0+1+r_0+r_{t+1})}}
{\pi^{n(n+1+r_0+r_{t+1})+\frac{n(n-1)}{2}t}}
(2\pi)^{-\frac{n t}{2}}
\frac{\prod_{\alpha=1}^t c_{\alpha}^{-n R_{\alpha,N}+\frac{n^2}{2}}}
{\prod_{k=1}^n (R_0-k)!}\tau^{-\frac{n(n+1)}{2}}      \\
&\times \prod_{1\leq i<j\leq n}
\big| \hat{z}_i-\hat{z}_j \big|^2
e^{-\frac{1}{\tau}{\rm Tr}\big( \hat{Z}\hat{Z}^* \big)
-\frac{1}{2}\sum_{\alpha=1}^t  
\frac{c_{\alpha}}{f_{\alpha}^2}\Big(
\overline{z_0-a_{\alpha}}^2
{\rm Tr}\big( \hat{Z}^2 \big)+(z_0-a_{\alpha})^2
{\rm Tr}\big( \hat{Z}^* \big)^2
\Big)}
\big( 1+O(N^{-1}) \big).
\end{aligned}
\end{equation}
\end{small}
In view of \eqref{parameter2}, rewrite
\begin{equation*}
\begin{aligned}
&\overline{a}_{\alpha}\hat{Z}+a_{\alpha}\hat{Z}^*
=\overline{a_{\alpha}-z_0}\hat{Z}+(a_{\alpha}-z_0)\hat{Z}^*
+\overline{z}_0\hat{Z}+z_0\hat{Z}^*,
\\
&\sum_{\alpha=1}^t \frac{c_{\alpha}}{f_{\alpha}^2}
\big( \overline{a}_{\alpha}\hat{Z}+a_{\alpha}\hat{Z}^* \big)
=P_1\big( \overline{z}_0\hat{Z}+z_0\hat{Z}^* \big)
+\overline{P}_0\hat{Z}+P_0\hat{Z}^*,
\end{aligned}
\end{equation*}
expand the square of above, we have
\begin{small}
\begin{equation}\label{ZW transform}
\begin{aligned}
&
-\frac{1}{2}\sum_{\alpha=1}^t  
\frac{c_{\alpha}}{f_{\alpha}^2}
\Big(
\overline{z_0-a_{\alpha}}^2
{\rm Tr}\big( \hat{Z}^2 \big)+(z_0-a_{\alpha})^2
{\rm Tr}\big( \hat{Z}^* \big)^2
\Big) -\frac{1}{2P_1}{\rm Tr}\Big(
\sum_{\alpha=1}^t   \frac{c_{\alpha}}{f_{\alpha}^2}
\big( \overline{a}_{\alpha}\hat{Z}+a_{\alpha}\hat{Z}^* \big) \Big)^2
\\  
&
 -\frac{1}{\tau}{\rm Tr}\big( \hat{Z}\hat{Z}^* \big)+\frac{1}{2}\sum_{\alpha=1}^t
 \frac{c_{\alpha}}
 {f_{\alpha}^2}{\rm Tr}\big( \big( \overline{a}_{\alpha}\hat{Z}+a_{\alpha}\hat{Z}^* \big)^2 \big)
 =-\frac{1}{2P_1}
{\rm Tr}\big( \overline{P}_0\hat{Z}+P_0\hat{Z}^* \big)^2.
\end{aligned}
\end{equation}
\end{small}
Changing variables like $\hat{Z}\rightarrow -\frac{\sqrt{P_1}}{\overline{P}_0}\hat{Z}$, $x\rightarrow \sqrt{P_1}x$ in \eqref{erfc integral}, $Y\rightarrow P_1^{-\frac{1}{4}}Y$ in \eqref{I0}, and combining
 \eqref{RNndelta}, \eqref{INdeltaexpan}, \eqref{J2Nestimation}, \eqref{J1Nform}-\eqref{Gamma integral}, \eqref{Ualpha integral}, \eqref{diagonal simplfy}, \eqref{erfc integral}, \eqref{eNf0 DNn} and \eqref{ZW transform}, we can get
\begin{small}
\begin{equation*}
\begin{aligned}
&\frac{1}{\big( N|P_0|^2P_1^{-1} \big)^n}
R_N^{(n)}\big(
A_0;z_0-\frac{\sqrt{P_1}\hat{z}_1}{\overline{P_0}\sqrt{N}}
,\cdots,
z_0-\frac{\sqrt{P_1}\hat{z}_n}{\overline{P_0}\sqrt{N}}
\big)
\\
&
=\frac{\prod_{1\leq i<j\leq n}\big|\hat{z}_i-\hat{z}_j\big|^2}
{\pi^{n^2+n}}
e^{-\frac{1}{2}\sum_{i=1}^n\big( \hat{z}_i+\overline{\hat{z}}_i \big)^2}
\prod_{i=1}^n \mathrm{IE}_{r_0-1}\big( -\hat{z}_i-\overline{\hat{z}_i} \big)
\\
&\times \int_{\mathbb{C}^{n^2}}
\big(
\det(YY^*)
\big)^{r_0}
\exp\Big\{
-\frac{1}{2} {\rm Tr}(YY^*)^2 
-
{\rm Tr}\big(
Y^*Y\hat{Z}
\big)-{\rm Tr}\big(
YY^*\hat{Z}^*
\big)
\Big\}{\rm d}Y+O\big(
N^{-\frac{1}{4}}
\big).
\end{aligned}
\end{equation*}
\end{small}
where $\mathrm{IE}_{n}(z)$ is defined in \eqref{IEF}. By the singular value decomposition:
\begin{small}
\begin{equation}\label{singularcorre}
Y=U\sqrt{R}V,\quad
R={\rm diag}\left( r_1,\cdots,r_n \right),\quad
r_1\geq \cdots \geq r_n \geq 0,
\end{equation}
\end{small}
 the Jacobian reads 
\begin{equation}\label{singularjacobian}
{\rm d}Y=\pi^{n^2}\Big( \prod_{i=1}^{n-1}i! \Big)^{-2}
\prod_{1\leq i<j\leq n}(r_{j}-r_{i})^2 
{\rm d}R{\rm d}U{\rm d}V,
\end{equation}
where 
$U$ and $ V$ are chosen from the unitary group  $\mathcal{U}(n)$ with the Haar measure. 
 Use  Harish-Chandra-Itzykson-Zuber integration formula  \cite{HC,IZ}  (see also \cite{Me})
\begin{equation}\label{HCIZ}
\int_{\mathcal{U}(n)}\exp\!\left\{
{\rm Tr}\left( AUBU^* \right)
\right\}{\rm d}U=
\frac{\det\big( [ e^{a_i b_j} ]_{i,j=1}^n \big)}
{\prod_{1\leq i<j\leq n}(a_{j}-a_{i}) (b_{j}-b_{i})} \prod_{i=1}^{n-1}i!,
\end{equation}
with  $A={\rm diag}\left( a_1,\cdots,a_n \right)$ and  $B={\rm diag}\left( b_1,\cdots,b_n \right)$, and also
Andr\'{e}ief's   integration formula \cite{And}  
\begin{small}
\begin{equation}\label{Cauchybinet}
 \int  \det\big( [ f_i(x_j) ]_{i,j=1}^n \big)
\det\big( [ g_i(x_j) ]_{i,j=1}^n \big) \prod_{i=1}^n{\rm d}\mu(x_i)=
n!\det\Big( \Big[ \int f_i(x)g_j(x){\rm d}\mu(x) \Big]_{i,j=1}^n \Big),
\end{equation}
\end{small}
we obtain 
\begin{small}
\begin{equation*}
\begin{aligned}
\frac{\prod_{1\leq i<j\leq n}\big|\hat{z}_i-\hat{z}_j\big|^2}
{\pi^{n^2+n}}
\int_{\mathbb{C}^{n^2}}
&\big(
\det(YY^*)
\big)^{r_0}
\exp\Big\{
-\frac{1}{2} {\rm Tr}(YY^*)^2 
-
{\rm Tr}\big(
Y^*Y\hat{Z}
\big)-{\rm Tr}\big(
YY^*\hat{Z}^*
\big)
\Big\}{\rm d}Y
 \\
&=
\det\left( \Big[
\sqrt{\frac{2}{\pi}}
\Gamma(r_0+1) 
e^{\frac{1}{2}
\big(
\hat{z}_i+\overline{\hat{z}}_j
 \big)^2}
\mathrm{IE}_{r_0}\left( 
\hat{z}_i+\overline{\hat{z}}_j
\right)
\Big]_{i,j=1}^n
\right).
\end{aligned}
\end{equation*}
\end{small}
 
After Step 1, Step 2 and Step 3,  we thus  give a complete proof for $z_0\not=0$ in Theorem \ref{2-complex-correlation}.
\end{proof}

\section{Proofs  of Propositions \ref{fTaylor}, \ref{gTaylor} and Theorem \ref{2-complex-correlation}: $z_0=0$}\label{propproofs}

 \subsection{Proofs of Propositions \ref{fTaylor} and \ref{gTaylor}}
\begin{proof}[Proof of Proposition \ref{fTaylor}]
 With the notation $f_{\alpha}$ in \eqref{falpha},    
 noticing the  integration domain   due to \eqref{SYUT}    the decomposition of     $T_{\alpha}$   as  sum of  a diagonal matrix $\sqrt{T_{{\rm d},\alpha}}$ and a strictly upper triangular  matrix ${T_{{\rm u},\alpha}}$ as in \eqref{Talpha 2}, 
  introduce new matrix variables 
\begin{equation}\label{Tdalpha1}
T_{{\rm d},\alpha}=\frac{\tau c_{\alpha}}{f_{\alpha}}
\mathbb{I}_n+\widetilde{T}_{{\rm d},\alpha},
\end{equation}
It's easy to obtain
\begin{equation}\label{Talphaexpansion1}
c_{\alpha}{\rm Tr}\log T_{{\rm d},\alpha}-\frac{f_{\alpha}}{\tau}{\rm Tr}(T_{{\rm d},\alpha})
=
nc_{\alpha}\big(
\log \frac{\tau c_{\alpha}}{f_{\alpha}}-1
\big)-\frac{f_{\alpha}^2}{2\tau^2 c_{\alpha}}{\rm Tr}\widetilde{T}_{{\rm d},\alpha}^2
+O\big(
\|
\widetilde{T}_{{\rm d},\alpha}
\|^3
\big),
\end{equation}
\begin{equation*}
\sqrt{T_{{\rm d},\alpha}}=\sqrt{\frac{\tau c_{\alpha}}{f_{\alpha}}}
\left(
\mathbb{I}_n+\frac{f_{\alpha}}{2\tau c_{\alpha}}\widetilde{T}_{{\rm d},\alpha}
+O\big(
\|
\widetilde{T}_{{\rm d},\alpha}
\|^2
\big)
\right),
\end{equation*}
\begin{equation*}
\begin{aligned}
T_{\alpha}T_{\alpha}^*
&=\frac{\tau c_{\alpha}}{f_{\alpha}}\mathbb{I}_n+\widetilde{T}_{{\rm d},\alpha}+
\sqrt{\frac{\tau c_{\alpha}}{f_{\alpha}}}
\left(
\mathbb{I}_n+\frac{f_{\alpha}}{2\tau c_{\alpha}}\widetilde{T}_{{\rm d},\alpha}
+O\big(
\|
\widetilde{T}_{{\rm d},\alpha}
\|^2
\big)
\right)
T_{{\rm u},\alpha}^*     
\\
&+\sqrt{\frac{\tau c_{\alpha}}{f_{\alpha}}}
T_{{\rm u},\alpha}
\left(
\mathbb{I}_n+\frac{f_{\alpha}}{2\tau c_{\alpha}}\widetilde{T}_{{\rm d},\alpha}
+O\big(
\|
\widetilde{T}_{{\rm d},\alpha}
\|^2
\big)\right)
+T_{{\rm u},\alpha}T_{{\rm u},\alpha}^*,
\end{aligned}
\end{equation*}
from which we have for $i<j$,
\begin{equation}\label{Talphaexpansion4}
\left(
\sum_{\alpha=1}^t a_{\alpha}T_{\alpha}T_{\alpha}^*
\right)_{i,j}
=\sum_{\alpha=1}^t a_{\alpha}\sqrt{\frac{\tau c_{\alpha}}{f_{\alpha}}}
t_{i,j}^{(\alpha)}
+O\left(
\| T_{{\rm u},\alpha} \|^2+\| T_{{\rm u},\alpha} \|
\| \widetilde{T}_{{\rm d},\alpha} \|
\right),
\end{equation} 
where \eqref{Talpha 2} has been used, and
\begin{equation}\label{Talphaexpansion5}
{\rm Tr}\big(
\hat{Z}T_{\alpha}T_{\alpha}^*
\big)
=\frac{\tau c_{\alpha}}{f_{\alpha}}{\rm Tr}\big(
\hat{Z}
\big)+
{\rm Tr}\big(
\hat{Z}\widetilde{T}_{{\rm d},\alpha}
\big)+O\left(
\|
T_{{\rm u},\alpha}
\|^2
\right).
\end{equation}
At the same time,
\begin{equation*}
\left(
z_0Q_0Q_0^*+Q_{t+1}A_{t+1}Q_{t+1}^*
\right)_{i,j}
=
z_0
\left(
Q_0Q_0^*
\right)_{i,j}
+O\left(
\| Q_{t+1} \|^2
\right).
\end{equation*}
Moreover, we see from the boundary condition $P_{00}(z_0)=1$ that  the restriction condition in  \eqref{Omega} becomes
\begin{equation*}
\sum_{\alpha=1}^t \widetilde{T}_{{\rm d},\alpha}+
\sum_{\alpha=1}^t \sqrt{\frac{\tau c_{\alpha}}{f_{\alpha}}}
T_{{\rm u},\alpha}+
\sum_{\alpha=1}^t \sqrt{\frac{\tau c_{\alpha}}{f_{\alpha}}}
T_{{\rm u},\alpha}^*+S\leq 0,
\end{equation*}
where
\begin{equation}\label{S}
\begin{aligned}
S&=\sum_{\alpha=1}^t T_{{\rm u},\alpha}T_{{\rm u},\alpha}^*
+Q_0Q_0^*+Q_{t+1}Q_{t+1}^*   \\
&+\sum_{\alpha=1}^t T_{{\rm u},\alpha}\Big(
\sqrt{T_{{\rm d},\alpha}}-
\sqrt{\frac{\tau c_{\alpha}}{f_{\alpha}}}\mathbb{I}_n
\Big)+
\sum_{\alpha=1}^t \Big(
\sqrt{T_{{\rm d},\alpha}}-
\sqrt{\frac{\tau c_{\alpha}}{f_{\alpha}}}\mathbb{I}_n
\Big)T_{{\rm u},\alpha}^*.
\end{aligned}
\end{equation}
Here the first line contains all diagonal parts of $S$, while the second line only contains the non-diagonal parts.

 Let  $S^{({\rm d})}$ be a  diagonal  matrix extracted   from the diagonal part of $S$ given  in \eqref{S},  introduce a new diagonal matrix   $\Gamma_1$ and make change of variables  from 
 $\widetilde{T}_{{\rm d},1}, \widetilde{T}_{{\rm d},2}, \cdots, \widetilde{T}_{{\rm d},t}$ to 
\begin{equation}\label{matrix transformations1}
\begin{aligned} 
\Gamma_1:=\sum_{\alpha=1}^t    \widetilde{T}_{{\rm d},\alpha}+S^{({\rm d})},
\quad \widetilde{T}_{{\rm d},2},\, \ldots, \widetilde{T}_{{\rm d},t}.
\end{aligned}
\end{equation}
Note that $S^{({\rm d})}$ doesn't depend on $\widetilde{T}_{{\rm d},\alpha}$, so the above  transformation is linear and
\begin{equation*}
S^{({\rm d})}=(Q_0Q_0^*)^{({\rm d})}+
O\big(
\| T_{{\rm u}} \|^2+\| Q_{t+1} \|^2
\big).
\end{equation*}
We will see that  the typical size for  $Q_{t+1}, \widetilde{T}_{{\rm d},2},\, \ldots, \widetilde{T}_{{\rm d},t}$ are $N^{-1/2}$   while it  is  $N^{-1}$ for  $\Gamma_1$ and it  is  $N^{-\frac{1}{4}}$ for  $Q_0$. So we will do Taylor expansions up to some proper orders  accordingly.

Also let  $S^{({\rm off},{\rm u})}$ be a strictly upper triangular matrix extracted from the strictly upper triangular part of $S$ given  in \eqref{S},  introduce a new strictly upper triangular matrix   $G_1$ and make change of variables  from 
 $T_{{\rm u},1}, T_{{\rm u},2}, \cdots, T_{{\rm u},t}$ to 
\begin{equation}\label{matrix transformations2}
\begin{aligned}
&G_1:=\sum_{\alpha=1}^t  \sqrt{\frac{\tau c_{\alpha}}{f_{\alpha}}}
T_{{\rm u},\alpha}+S^{({\rm off},{\rm u})},
\quad T_{{\rm u},2},\, \ldots,T_{{\rm u},t}.
\end{aligned}
\end{equation}
We will see that  the typical size for $T_{{\rm u},2}, \cdots, T_{{\rm u},t}$  are $N^{-1/2}$   while it  is  $N^{-1}$ for  $G_1$. So we will do Taylor expansions up to some proper orders  accordingly.

As to the change of variables   \eqref{matrix transformations2}, the Jacobian determinant reads
\begin{equation*}
\det\Big(
\frac{\partial T_{{\rm u},1}}{\partial G_1}
\Big)
=\left(
\frac{f_1}{\tau c_1}
\right)^{\frac{n(n-1)}{2}}+
O\big(
\| T_{{\rm u},1} \|+\| \widetilde{T}_{{\rm d}} \|
\big).
\end{equation*}
At the same time, 
\begin{equation}\label{Tdestimate0}
\| \widetilde{T}_{{\rm d}} \|=O\big(
\| \widetilde{T}_{{\rm d}} \|_2+\|\Gamma_1\|+\| Q_0 \|^2
+\| T_{{\rm u}} \|+\| Q_{t+1} \|
\big),
\end{equation}
where the last two terms  can be replaced by  $\| T_{{\rm u}} \|^2+\| Q_{t+1} \|^2$ for exactness,
and from 
\begin{equation*}
S^{({\rm off},{\rm u})}=(Q_0Q_0^*)^{({\rm off},{\rm u})}+O\Big(
\| T_{{\rm u}} \|\big(
\| T_{{\rm u}} \|+\| \widetilde{T}_{{\rm d}} \|
\big)+\| Q_{t+1} \|^2
\Big),
\end{equation*}
we have from \eqref{matrix transformations1} that 
\begin{equation*}
\| T_{{\rm u}} \|=O\big(
\| G_1 \|+\|T_{{\rm u}}\|_2+\| Q_0 \|^2
+\| Q_{t+1} \|
\big)+\| T_{{\rm u}} \|O\big(
\| T_{{\rm u}} \|+\| \widetilde{T}_{{\rm d}} \|
\big).
\end{equation*}
Also from \eqref{ANdelta} we have 
$$ \| T_{{\rm u}} \|+\| \widetilde{T}_{{\rm d}} \|
=O\big( \sqrt{\delta} \big),
 $$
as $\delta>0$ is sufficiently small, we have
 \begin{equation}\label{Tuestimate}
\| T_{{\rm u}} \|=O\big(
\| G_1 \|+\|T_{{\rm u}}\|_2+\| Q_0 \|^2
+\| Q_{t+1} \|
\big).
\end{equation}
Substitute \eqref{Tuestimate} into \eqref{Tdestimate0}, we have
\begin{equation}\label{Tdestimate1}
\| \widetilde{T}_{{\rm d}} \|=O\big(
\Omega_{{\rm error}}
\big),
\end{equation}
where
\begin{equation}\label{Omegaerror}   
\Omega_{{\rm error}}:=\| \widetilde{T}_{{\rm d}} \|_2+\|\Gamma_1\|
+\| G_1 \|+\|T_{{\rm u}}\|_2
+\| Q_0 \|^2+\| Q_{t+1} \|,
\end{equation}
and further
\begin{equation}\label{Jacobiandet}
\det\Big(
\frac{\partial T_{{\rm u},1}}{\partial G_1}
\Big)
=\left(
\frac{f_1}{\tau c_1}
\right)^{\frac{n(n-1)}{2}}+O(\Omega_{{\rm error}}).
\end{equation}
In view of \eqref{fTY} and \eqref{hQrewrite}, we get 
\begin{equation*}
\frac{|z_0|^2}{\tau}\Big(
\sum_{\alpha=1}^t {\rm Tr}(T_{\alpha}T_{\alpha}^*)+
{\rm Tr}(Q_{0}Q_{0}^*)+{\rm Tr}(Q_{t+1}Q_{t+1}^*)
\Big)=\frac{n|z_0|^2}{\tau}+\frac{|z_0|^2}{\tau}{\rm Tr}(\Gamma_1).
\end{equation*}

In view of \eqref{fTY}, \eqref{hQrewrite} and \eqref{Talphaexpansion4}, from transformations \eqref{matrix transformations1} and \eqref{matrix transformations2} we have
\begin{small}
\begin{equation*}
\begin{aligned}
&-\sum_{\alpha=1}^t
f_{\alpha}{\rm Tr}\big(
T_{{\rm u},\alpha}T_{{\rm u},\alpha}^*
\big)+\sum_{1\leq i < j\leq n}
 \left|
 \left(
 \sum_{\alpha=1}^t a_{\alpha}T_{\alpha}T_{\alpha}^*
+z_0
Q_0Q_0^*+
Q_{t+1}A_{t+1}Q_{t+1}^*
 \right)_{i,j} 
 \right|^2       \\
 &=-\frac{f_1^2}{\tau c_1}
 {\rm Tr}\left(
 (Q_0Q_0^*)^{({\rm off},{\rm u})}+\sum_{\alpha=2}^t
 \sqrt{\frac{\tau c_{\alpha}}{f_{\alpha}}}T_{{\rm u},\alpha}
 \right)
 \left(
 (Q_0Q_0^*)^{({\rm off},{\rm u})}+\sum_{\alpha=2}^t
 \sqrt{\frac{\tau c_{\alpha}}{f_{\alpha}}}T_{{\rm u},\alpha}
 \right)^*  -\sum_{\alpha=2}^t
f_{\alpha}{\rm Tr}\big(
T_{{\rm u},\alpha}T_{{\rm u},\alpha}^*
\big)             \\
&     
+\sum_{1\leq i < j\leq n}
 \left|
(z_0-a_1)  (Q_0Q_0^*)^{({\rm off},{\rm u})}_{i,j}
+\sum_{\alpha=2}^t
\sqrt{\frac{\tau c_{\alpha}}{f_{\alpha}}}
(a_{\alpha}-a_1)
t_{i,j}^{(\alpha)}
 \right|^2+O(\Omega_{{\rm error}}^3+\|G_1\|\Omega_{{\rm error}}),
\end{aligned}
\end{equation*}
\end{small}
and
\begin{small}
\begin{equation}\label{diagonal1}
\begin{aligned}
&-\sum_{\alpha=1}^t\frac{f_{\alpha}^2}{2\tau^2 c_{\alpha}}
{\rm Tr}
\left(
\widetilde{T}_{{\rm d},\alpha}^2
\right)
+\frac{|z_0|^2}{\tau}{\rm Tr}\left(
\sum_{\alpha=1}^t\widetilde{T}_{{\rm d},\alpha}
+S^{({\rm d})}
\right)
+\frac{N^{-\frac{1}{2}}}{\tau}
\sum_{\alpha=1}^t\left(
\overline{a}_{\alpha}{\rm Tr}\big(
\hat{Z}\widetilde{T}_{{\rm d},\alpha}
\big)
+a_{\alpha}{\rm Tr}\big(
\hat{Z}^*\widetilde{T}_{{\rm d},\alpha}
\big)
\right)              \\
&=\frac{|z_0|^2}{\tau}{\rm Tr}(\Gamma_1)
-\frac{f_{1}^2}{2\tau^2 c_{1}}{\rm Tr}
\left(
 (Q_0Q_0^*)^{({\rm d})}+\sum_{\alpha=2}^t
\widetilde{T}_{{\rm d},\alpha}
 \right)^2
 -\sum_{\alpha=2}^t\frac{f_{\alpha}^2}{2\tau^2 c_{\alpha}}
{\rm Tr}\big(\widetilde{T}_{{\rm d},\alpha}^2  \big)  
+\frac{N^{-\frac{1}{2}}}{\tau}\sum_{\alpha=2}^t{\rm Tr}
\big((\overline{a}_{\alpha}\hat{Z}+a_{\alpha}\hat{Z}^*)
\widetilde{T}_{{\rm d},\alpha}\big)    \\
&-\frac{N^{-\frac{1}{2}}}{\tau}{\rm Tr}\Big(
(\overline{a}_{1}\hat{Z}+a_{1}\hat{Z}^*)
\big(
(Q_0Q_0^*)^{({\rm d})}+\sum_{\alpha=2}^t \widetilde{T}_{{\rm d},\alpha}
\big)
\Big)+O\Big(
\big(
\Omega_{{\rm error}}+N^{-\frac{1}{2}}
\big)
\big(
\Omega_{{\rm error}}^2+\|\Gamma_1\|
\big)
\Big).
\end{aligned}
\end{equation}
\end{small}
Also, the integration  region now becomes simply
\begin{equation}\label{SimpleRegion}
\Gamma_1+G_1+G_1^*\leq 0.
\end{equation}

In view of \eqref{logAalpha decompose}, noticing \eqref{HAalpha}, we obtain 
\begin{small}
\begin{equation}\label{logdetHA}
\begin{aligned}
\log\det\left(
\mathbb{I}_{2n}+\sqrt{\gamma_N}
N^{-\frac{1}{2}}\widehat{A}_{\alpha}
\right)
=N^{-\frac{1}{2}}{\rm Tr}\big(
\widehat{A}_{\alpha}
\big)-\frac{1}{2}N^{-1}{\rm Tr}\big(
\widehat{A}_{\alpha}^2
\big)
+O\big(
N^{-\frac{3}{2}}
\big).
\end{aligned}
\end{equation}
\end{small}

Combining \eqref{logAalpha decompose}, \eqref{HAalpha} and \eqref{logdetHA}, and introducing $P_0$ defined in \eqref{parameter}, we have
\begin{small}
\begin{equation}\label{finY}
\begin{aligned}
&{\rm Tr}(YY^*)-\sum_{\alpha=1}^t
c_{\alpha}\log\det(A_{\alpha})=
-\frac{n^2}{N}-n\sum_{\alpha=1}^tc_{\alpha}\log f_{\alpha}
\\
&+\frac{N^{-1}}{2}
\sum_{\alpha=1}^t\frac{c_{\alpha}}{f_{\alpha}^2}
\left(
\overline{z_0-a_{\alpha}}^2{\rm Tr}\big(
\hat{Z}^2
\big)+
(z_0-a_{\alpha})^2{\rm Tr}\big(
\hat{Z}^*
\big)^2
\right)    \\
&-N^{-\frac{1}{2}}\sum_{\alpha=1}^t\frac{c_{\alpha}}{f_{\alpha}}
\left(
\overline{z_0-a_{\alpha}}
{\rm Tr}\big(
\hat{Z}
\big)+
(z_0-a_{\alpha})
{\rm Tr}\big(
\hat{Z}^*
\big)
\right)   
 +\frac{P_1}{2} {\rm Tr}(YY^*)^2 \\
&-N^{-\frac{1}{2}}\left(
\overline{P_0}{\rm Tr}\big(
Y^*Y\hat{Z}
\big)+P_0{\rm Tr}\big(
YY^*\hat{Z}^*
\big)
\right)+O\big(
N^{-\frac{3}{2}}+\| Y \|^6+N^{-1}\| Y \|^2+N^{-\frac{1}{2}}\| Y \|^4
\big).
\end{aligned}
\end{equation}
\end{small}

Combining \eqref{fTY}, \eqref{Talphaexpansion1}, \eqref{Talphaexpansion5}, \eqref{Tuestimate}-\eqref{diagonal1} and \eqref{finY}, we thus obtain the desired result. 
\end{proof}

\begin{proof}[Proof of Proposition \ref{gTaylor}]
For $g(T,Y,Q)$ defined in \eqref{gYUT},  we write down  leading terms  for five relevant  factors while the estimate for  the last determinant   in \eqref{gYUT}  is  non-trivial.
 \begin{itemize}
 \item[(1)]For the first factor in  \eqref{gYUT},  by \eqref{matrix transformations1} and \eqref{matrix transformations2},
 \begin{equation}\label{detTQ}
\Big(
\det\big(
\mathbb{I}_n-\sum_{\alpha=1}^tT_{\alpha}T_{\alpha}^*
-Q_0Q_0^*-Q_{t+1}Q_{t+1}^*
\big)
\Big)^{R_0-n}=\Big(\det\big(
-\Gamma_1-G_1-G_1^*
\big)\Big)^{R_0-n},
\end{equation}
  \item[(2)]For the second factor in  \eqref{gYUT}, by \eqref{zi},
 \begin{equation*}
\det\begin{bmatrix}
\sqrt{\gamma_N}Z & -Y^*   \\
Y  &   \sqrt{\gamma_N}Z^*
\end{bmatrix}=
|z_0|^{2n}+O\big(
N^{-\frac{1}{2}}+\| Y \|
\big).  
\end{equation*} 
  \item[(3)]For the third factor in  \eqref{gYUT} with $\alpha=1,\cdots,t,$, it is easy to see from \eqref{Tdalpha1} and \eqref{Omegaerror} that
 \begin{equation}\label{tjjalpha}
 \prod_{\alpha=1}^t\prod_{j=1}^n
\big(
t_{j,j}^{(\alpha)}
\big)^{R_{\alpha,N}-j}=
\prod_{\alpha=1}^t\prod_{j=1}^n
\Big(
   \frac{\tau c_{\alpha}}{f_{\alpha}}
+O\big(\Omega_{\rm error}
\big)\Big)^{R_{\alpha,N}-j}.  
\end{equation}
 \item[(4)]For the fourth factor in  \eqref{gYUT}, by \eqref{zi} and \eqref{A alpha}
 \begin{equation}\label{detAalpha}
 \prod_{\alpha=1}^t  \big(\det(A_{\alpha})\big)^{R_{\alpha,N}-n}=
 \prod_{\alpha=1}^t\Big(f_{\alpha}^n
+O\big(
N^{-\frac{1}{2}}+\| Y \|
\big)\Big)^{R_{\alpha,N}-n}.
\end{equation}

  \end{itemize}
  
The remaining part of this subsection  is devoted to    the last determinant   in \eqref{gYUT}.

Recalling $\hat{L}_1$ and $\hat{L}_2$ defined in    \eqref {L1hat} and \eqref{L0hat}, since all $a_{\alpha}\neq z_0$, $A_{a_{\alpha}}$ is  invertible when $N$ is sufficiently large. Rewrite
\begin{equation*}
\det\!\bigg(
\widehat{L}_1+\sqrt{\gamma_N}\widehat{L}_2
\bigg)=\left[\begin{smallmatrix}
P_{1,1} & O\big( \|\widetilde{Q}_{t+1}\| \big)   \\
O\big( \|\widetilde{Q}_{t+1}\| \big) &
\widehat{B}_{t+1}+O\big( \|\widetilde{Q}_{t+1}\|^2 \big)
 \end{smallmatrix}
\right],
\end{equation*}
where
\begin{small}
\begin{equation*}
P_{1,1}=\left[\begin{smallmatrix}
A_1\otimes \mathbb{I}_n &&\\
& \ldots & \\
&& A_t\otimes \mathbb{I}_n
 \end{smallmatrix}
\right]+
\sqrt{\gamma_N}
\left[\begin{smallmatrix}
\mathbb{I}_{2n}\otimes T_1^*\\
\vdots \\
 \mathbb{I}_{2n}\otimes T_t^* 
 \end{smallmatrix}
\right]
\left[\begin{smallmatrix}
\left[\begin{smallmatrix}
a_1\mathbb{I}_n & \\ & \overline{a}_{1}\mathbb{I}_n
\end{smallmatrix}\right]\otimes T_1 &
,\cdots, &
\left[\begin{smallmatrix}
a_{t}\mathbb{I}_n & \\ & \overline{a}_{t}\mathbb{I}_n
\end{smallmatrix}\right]\otimes T_t
\end{smallmatrix}\right],
\end{equation*}
\end{small}
change the order of the above matrix product and  we obtain 
 \begin{equation*}
\begin{aligned}
&\det(P_{1,1})=\Big(\prod_{\alpha=1}^t \det\big( A_{\alpha} \big) \Big)^n
\det\bigg(
\mathbb{I}_{2n^2}+\sqrt{\gamma_N} 
\sum_{\alpha=1}^t
\left( A_{\alpha}^{-1} \left[\begin{smallmatrix}
a_{\alpha}\mathbb{I}_n & \\ & \overline{a}_{\alpha}\mathbb{I}_n
\end{smallmatrix}\right] \right) \otimes
(T_{\alpha}T_{\alpha}^*)
\bigg),
\end{aligned}
\end{equation*}
Simple calculations show that
\begin{equation}\label{Zalpha plus inverse}
\Big(
Z_{\alpha}+Y^*\big(Z_{\alpha}^*\big)^{-1}Y
\Big)^{-1}=(z_0-a_{\alpha})^{-1}\mathbb{I}_n
+O\big( N^{-\frac{1}{2}}+\|Y\|^2 \big),
\end{equation}
\begin{equation}\label{Talpha Talphastar}
T_{\alpha}T_{\alpha}^*=\frac{\tau c_{\alpha}}{f_{\alpha}}\mathbb{I}_n
+O\big( 
\Omega_{\rm error}
 \big).
\end{equation}
Rewrite $\frac{a_{\alpha}}{z_0-a_{\alpha}}=
\frac{a_{\alpha}\overline{z_0-a_{\alpha}}}{f_{\alpha}}$, we have
\begin{small}
\begin{equation*}
\begin{aligned}
\mathbb{I}_{2n^2}
&+\sqrt{\gamma_N} 
\sum_{\alpha=1}^t
\left( A_{\alpha}^{-1} \left[\begin{smallmatrix}
a_{\alpha}\mathbb{I}_n & \\ & \overline{a}_{\alpha}\mathbb{I}_n
\end{smallmatrix}\right] \right) \otimes
(T_{\alpha}T_{\alpha}^*)=
\mathbb{I}_{2n^2}+
\\&
\sum_{\alpha=1}^t
\begin{bmatrix}
\frac{\tau c_{\alpha}a_{\alpha}\overline{z_0-a_{\alpha}}}{f_{\alpha}^2}
\mathbb{I}_{n^2} & \\
& 
\frac{\tau c_{\alpha}\overline{a}_{\alpha}(z_0-a_{\alpha})}{f_{\alpha}^2}
\mathbb{I}_{n^2}
\end{bmatrix}
+O\big( 
\Omega_{\rm error}
+N^{-\frac{1}{2}}+\|Y\|
 \big).
\end{aligned}
\end{equation*}
\end{small}
While from $P_{00}(z_0)=1$
\begin{equation*}
\begin{aligned}
1+\sum_{\alpha=1}^t \frac{\tau c_{\alpha}a_{\alpha}\overline{z_0-a_{\alpha}}}{f_{\alpha}^2}&=
1+\tau\sum_{\alpha=1}^t\frac{c_{\alpha}(a_{\alpha}-z_0+z_0)
\overline{z_0-a_{\alpha}}}{f_{\alpha}^2}   \\
&=1-\tau\sum_{\alpha=1}^t \frac{c_{\alpha}}{f_{\alpha}}
+\tau z_0\sum_{\alpha=1}^t \frac{c_{\alpha}\overline{z_0-a_{\alpha}}}{f_{\alpha}^2}=-\tau z_0\overline{P}_0,
\end{aligned}
\end{equation*}
where $P_0$ is defined in \eqref{parameter}. Combining \eqref{detAalpha} we have
\begin{equation}\label{detsum tildeAalpha}
\det\big(  
P_{1,1}
\big)   
=\tau^{2n^2}|z_0|^{2n^2}|P_0|^{2n^2}\prod_{\alpha=1}^t f_{\alpha}^{n^2}+
O\big( 
\Omega_{\rm error}
+N^{-\frac{1}{2}}+\|Y\|
 \big).
\end{equation} 
As $P_0\not=0$ and $z_0\not=0$, we have $ \det(P_{1,1})\not=0. $
and from $\widetilde{Q}_{t+1}=[Q_0,Q_{t+1}]$, $
\|\widetilde{Q}_{t+1}\|^2=O\big(
\|Q_0\|^2+\|Q_{t+1}\|
\big)=O\big(
\Omega_{\rm error}
\big).
$ Therefore,
\begin{equation*}
\det\!\bigg(
\widehat{L}_1+\sqrt{\gamma_N}\widehat{L}_2
\bigg)=\det(P_{1,1})\det\Big(
\widehat{B}_{t+1}+O\big(
\|\widetilde{Q}_{t+1}\|^2
\big)
\Big).
\end{equation*}
In view of \eqref{B0hat} and \eqref{Q convenience}, from elementary computations,
\begin{small}
\begin{equation}\label{detP22}
\det\Big(
\widehat{B}_{t+1}+O\big(
\|\widetilde{Q}_{t+1}\|^2
\big)
\Big)= 
\big| \det\big( z_0\mathbb{I}_{r_{t+1}}-A_{t+1} \big) \big|^{2n}
 \Big(
\big(\det(YY^*)\big)^{r_0}
 +O\Big(
\sum_{\alpha,\beta}\| Y \|^{\alpha}
\Omega_{\rm error}^{\beta}
\Big)
\Big),
\end{equation}
\end{small}
 where the finite non-negative integer pair $(\alpha,\beta)$ is such that
 \begin{equation}\label{M error condition}
\frac{1}{4}\alpha+\frac{1}{2}\beta\geq 
\frac{1}{2}n r_0+\frac{1}{4}.
\end{equation}
Combing \eqref{detsum tildeAalpha}-\eqref{detP22}, we have
\begin{small}
 \begin{equation*}
\begin{aligned}
\det\!\bigg(
\widehat{L}_1+\sqrt{\gamma_N}\widehat{L}_2
\bigg)&=\big| \det\big( z_0\mathbb{I}_{r_{t+1}}-A_{t+1} \big) \big|^{2n}
\Big(
\big(\det(YY^*)\big)^{r_0}
 +O\Big(
\sum_{\alpha,\beta}\| Y \|^{\alpha}
\Omega_{\rm error}^{\beta}
\Big)
\Big)
\\&\times\Big(
\tau^{2n^2}|z_0|^{2n^2}|P_0|^{2n^2}
\prod_{\alpha=1}^t
f_{\alpha}^{n^2}
 +O\big(
\Omega_{\rm error}+N^{-\frac{1}{2}}+\| Y \|
\big)
\Big),
\end{aligned}
\end{equation*}
\end{small}
where the finite non-negative integer pair $(\alpha,\beta)$ is such that
 \eqref{M error condition}.
\end{proof}

\subsection{Proof of Theorem \ref{2-complex-correlation}: $z_0= 0$}

The proof in  the case of $z_0=0$ is similar,  we only point out the different places from the case of $z_0\not=0$.

As $z_0=0$,  in view of $X_0={\rm diag}\left( A_0,0_{R_0} \right)$ in Definition \ref{GinU}, $R_0=N-r$  and 
\newline
$A_0={\rm diag}\left(a_1\mathbb{I}_{r_1},\cdots,a_t\mathbb{I}_{r_t},
0_{r_0},A_{t+1}
\right)$ below and in \eqref{A0 form}, there are $R_0+r_0$ 0s. Hence, without loss of generality we can simply set $r_0=0$, and further $Q_0=0$ in Section  \ref{sectnotation}. For {\bf Taylor expansion of $f(T,Y,Q)$}, We still have \eqref{Tdalpha1}-\eqref{Tdestimate1} and   \eqref{SimpleRegion} with $z_0=0$ and $Q_0=0$. But note that as $z_0=0$, linear term of $\Gamma_1$ does not exist on the exponent. This means that the typical size of $\Gamma_1$ is $N^{-\frac{1}{2}}$. Correspondingly, the typical size of $G_1$ is also $N^{-\frac{1}{2}}$, this is different from the case of $z_0\not=0$.
 
 Hence, we have
 \begin{small}
\begin{equation}\label{triangle1 z0equal0}
\begin{aligned}
&-\sum_{\alpha=1}^t
f_{\alpha}{\rm Tr}\big(
T_{{\rm u},\alpha}T_{{\rm u},\alpha}^*
\big)+\sum_{1\leq i < j\leq n}
 \left|
 \left(
 \sum_{\alpha=1}^t a_{\alpha}T_{\alpha}T_{\alpha}^*
 +
Q_{t+1}A_{t+1}Q_{t+1}^*
 \right)_{i,j} 
 \right|^2       \\
 &=-\frac{f_1^2}{\tau c_1}
 {\rm Tr}\left(
 G_1-\sum_{\alpha=2}^t
 \sqrt{\frac{\tau c_{\alpha}}{f_{\alpha}}}T_{{\rm u},\alpha}
 \right)
 \left(
 G_1-\sum_{\alpha=2}^t
 \sqrt{\frac{\tau c_{\alpha}}{f_{\alpha}}}T_{{\rm u},\alpha}
 \right)^*        
-\sum_{\alpha=2}^t
f_{\alpha}{\rm Tr}\big(
T_{{\rm u},\alpha}T_{{\rm u},\alpha}^*
\big)     \\      
&+\sum_{1\leq i < j\leq n}
 \left|
a_1(G_1)_{i,j}
+\sum_{\alpha=2}^t
\sqrt{\frac{\tau c_{\alpha}}{f_{\alpha}}}
(a_{\alpha}-a_1)
t_{i,j}^{(\alpha)}
 \right|^2   
 +O\left(
\Omega_{\rm error}^3
 \right),
\end{aligned}
\end{equation}
\end{small}
where
\begin{equation*}
\Omega_{\rm error}=\| G_1 \|+\| T_{\rm u} \|_2
+\| \Gamma_1 \|+\| \widetilde{T}_{\rm d} \|_2+\| Q_{t+1} \|.
\end{equation*}
In view of \eqref{Tdalpha1} and \eqref{Talphaexpansion5},
\begin{small}
\begin{equation}\label{diagonal1 z0equal0}
\begin{aligned}
&-\sum_{\alpha=1}^t\frac{f_{\alpha}^2}{2\tau^2 c_{\alpha}}
{\rm Tr}
\big(
\widetilde{T}_{{\rm d},\alpha}^2
\big)+\frac{1}{\tau}
N^{-\frac{1}{2}}
\sum_{\alpha=1}^t\Big(
\overline{a}_{\alpha}{\rm Tr}\big(
\hat{Z}\widetilde{T}_{{\rm d},\alpha}
\big)
+a_{\alpha}{\rm Tr}\big(
\hat{Z}^*\widetilde{T}_{{\rm d},\alpha}
\big)
\Big)
\\&
=
-\frac{f_{1}^2}{2\tau^2 c_{1}}{\rm Tr}
\Big(
 \Gamma_1-\sum_{\alpha=2}^t
\widetilde{T}_{{\rm d},\alpha}
\Big)^2
 -\sum_{\alpha=2}^t\frac{f_{\alpha}^2}{2\tau^2 c_{\alpha}}
{\rm Tr}\big( \widetilde{T}_{{\rm d},\alpha}^2 \big)   
+\frac{1}{\tau}
N^{-\frac{1}{2}}
\sum_{\alpha=2}^t{\rm Tr}
\Big(
\big( \overline{a}_{\alpha}\hat{Z}+a_{\alpha}\hat{Z}^* \big)  
\widetilde{T}_{{\rm d},\alpha}
\Big) 
\\&   
+\frac{1}{\tau}
N^{-\frac{1}{2}}
{\rm Tr}
\Big(
\big( \overline{a}_{1}\hat{Z}+a_{1}\hat{Z}^* \big) 
\big(
\Gamma_1-\sum_{\alpha=2}^t \widetilde{T}_{{\rm d},\alpha}
\big)
\Big)
+O
\Big(
\big(
N^{-\frac{1}{2}}+\Omega_{\rm error}
\big)^3
\Big).       
\end{aligned}
\end{equation}
\end{small}
Also we still have \eqref{finY} with $z_0=0$. Combining all together,
\begin{equation*}
f(T,Y)-f_0\big(\tau;\big\{\frac{\tau c_{\alpha}}{f_{\alpha}}\mathbb{I}_n\big\},0,0\big)
=F_0+F_1+O(F_2),
\end{equation*}
where
\begin{equation*}
F_0=\frac{n^2}{N}       
-\frac{N^{-1}}{2}
\sum_{\alpha=1}^t\frac{c_{\alpha}}{f_{\alpha}^2}
\left(
\overline{a_{\alpha}}^2
{\rm Tr}\big(
\hat{Z}^2
\big)+
a_{\alpha}^2{\rm Tr}\big(
\hat{Z}^*
\big)^2
\right)  ,
\end{equation*}

\begin{small}
\begin{equation}\label{f1 0}
\begin{aligned}
&F_1=
-\frac{f_{1}^2}{2\tau^2 c_{1}}{\rm Tr}
\Big(
 \Gamma_1-\sum_{\alpha=2}^t
\widetilde{T}_{{\rm d},\alpha}
 \Big)^2
  -\sum_{\alpha=2}^t\frac{f_{\alpha}^2}{2\tau^2 c_{\alpha}}
  {\rm Tr}\big( \widetilde{T}_{{\rm d},\alpha}^2  \big)
-\frac{1}{\tau}{\rm Tr}\big(
Q_{t+1}
A_{t+1}^*
A_{t+1}
Q_{t+1}^*  \big)     \\
&+\frac{N^{-\frac{1}{2}}}{\tau}\left(
\sum_{\alpha=2}^t{\rm Tr}
\Big(
\big( \overline{a}_{\alpha}\hat{Z}+a_{\alpha}\hat{Z}^* \big)  
\widetilde{T}_{{\rm d},\alpha}
\Big)
+{\rm Tr}
\Big(
\big( \overline{a}_{1}\hat{Z}+a_{1}\hat{Z}^* \big) 
\big(
\Gamma_1-\sum_{\alpha=2}^t \widetilde{T}_{{\rm d},\alpha}
\big)
\Big)
\right)  \\
&-\sum_{\alpha=2}^t
\frac{f_{\alpha}}{\tau}{\rm Tr}\big(
T_{{\rm u},\alpha}T_{{\rm u},\alpha}^*
\big)-\frac{f_1^2}{\tau^2 c_1}
 {\rm Tr}\Big(
 G_1-\sum_{\alpha=2}^t
 \sqrt{\frac{\tau c_{\alpha}}{f_{\alpha}}}T_{{\rm u},\alpha}
 \Big)
 \Big(
 G_1-\sum_{\alpha=2}^t
 \sqrt{\frac{\tau c_{\alpha}}{f_{\alpha}}}T_{{\rm u},\alpha}
 \Big)^* 
\\&
 +\frac{1}{\tau}\sum_{1\leq i < j\leq n}
 \Big|
a_1(G_1)_{i,j}
+\sum_{\alpha=2}^t
\sqrt{\frac{\tau c_{\alpha}}{f_{\alpha}}}
(a_{\alpha}-a_1)
t_{i,j}^{(\alpha)}
 \Big|^2    
 \\& -\frac{P_1}{2} {\rm Tr}(YY^*)^2 
+N^{-\frac{1}{2}}\left(
\overline{P_0}{\rm Tr}\big(
Y^*Y\hat{Z}
\big)+P_0{\rm Tr}\big(
YY^*\hat{Z}^*
\big)
\right),
\end{aligned}
\end{equation}
\end{small}
and
\begin{small}
\begin{equation*}
\begin{aligned}
F_2=N^{-\frac{3}{2}}+\| Y \|^6+N^{-1}\| Y \|^2+N^{-\frac{1}{2}}\| Y \|^4
+ \big(
\Omega_{\rm error}+N^{-\frac{1}{2}}
\big)^3.
\end{aligned}
\end{equation*}
\end{small}

For {\bf Taylor expansion of $g(Y,U,T)$} defined in \eqref{gYUT}, we still have \eqref{detTQ}, \eqref{tjjalpha} and \eqref{detAalpha}, while as $z_0=0$,
\begin{small}
 \begin{equation*}
\det\begin{bmatrix}
\sqrt{\gamma_N}Z & -Y^*   \\
Y  &   \sqrt{\gamma_N}Z^*
\end{bmatrix}=\big(\det(YY^*)\big)^{n}+
O\Big(
\sum_{\alpha,\beta}\| Y \|^{\alpha}
N^{-\frac{\beta}{2}}
\Big),
\end{equation*}
\end{small}
 where $ \gamma_N=\frac{N}{N-n} $, the finite non-negative integer pair $(\alpha,\beta)$ is such that
\begin{small}
 \begin{equation}\label{M error condition 0}
\frac{\alpha}{4}+\frac{\beta}{2}\geq 
\frac{n^2}{2}+\frac{1}{4}.
\end{equation}
\end{small}

The remaining part is devoted to  the last determinant   in \eqref{gYUT}.

Recalling $\hat{L}_1$ and $\hat{L}_2$ defined in    \eqref {L1hat} and \eqref{L0hat},
since  $0$ is not an eigenvalue of $A_{t+1}$ and all $a_{\alpha}\neq 0$,  
$\widehat{L}_1$ is  invertible when $N$ is sufficiently large.  Together with the decomposition 
\begin{equation*}
\widehat{L}_2
=\left[\begin{smallmatrix}
\mathbb{I}_{2n}\otimes T_1^*\\
\vdots \\
 \mathbb{I}_{2n}\otimes T_t^* \\
 \left[\begin{smallmatrix}
\mathbb{I}_n \otimes  A_{t+1}Q_{t+1}^*  & \\ & \mathbb{I}_n \otimes  A_{t+1}^* Q_{t+1}^*
\end{smallmatrix}\right]
 \end{smallmatrix}
\right]
\left[\begin{smallmatrix}
\left[\begin{smallmatrix}
a_1\mathbb{I}_n & \\ & \overline{a}_{1}\mathbb{I}_n
\end{smallmatrix}\right]\otimes T_1 &
,\cdots, &
\left[\begin{smallmatrix}
a_{t}\mathbb{I}_n & \\ & \overline{a}_{t}\mathbb{I}_n
\end{smallmatrix}\right]\otimes T_t, &
\mathbb{I}_{2n}\otimes Q_{t+1}
\end{smallmatrix}\right],
\end{equation*} 
change the order of the above matrix product and  we obtain 
 \begin{small}
 \begin{equation}\label{detMrewrite0}
\begin{aligned}
&\det\!\bigg(
\widehat{L}_1+\sqrt{\gamma_N}\widehat{L}_2
\bigg)=\big|\det\big(A_{t+1}\big)\big|^{2n}
\Big( \prod_{\alpha=1}^t \det\big( A_{\alpha} \big) \Big)^n
\big( 1+O(\| Y \|+N^{-\frac{1}{4}}) \big) \\
&\times \det\bigg(
\mathbb{I}_{2n^2}+\sqrt{\gamma_N} \begin{bmatrix}
Q_{1,1} & Q_{1,2} \\ Q_{2,1} & Q_{2,2}
\end{bmatrix} +
O\big( \| Q_{t+1} \|^2 \big)
\bigg),
\end{aligned}
\end{equation}
 \end{small}
where 
\begin{small}
\begin{equation*}
\begin{bmatrix}
Q_{1,1} & Q_{1,2} \\ Q_{2,1} & Q_{2,2}
\end{bmatrix}:=\sum_{\alpha=1}^t
\left( A_{\alpha}^{-1} \left[\begin{smallmatrix}
a_{\alpha}\mathbb{I}_n & \\ & \overline{a}_{\alpha}\mathbb{I}_n
\end{smallmatrix}\right] \right) \otimes
(T_{\alpha}T_{\alpha}^*).
\end{equation*}
\end{small} 
Obviously,  by   \eqref{A alpha}, setting $Z_{\alpha}=\sqrt{\gamma_N}(Z-a_{\alpha}\mathbb{I}_n)$ we have  
\begin{small}
\begin{equation*}
\begin{aligned}
&Q_{1,1}=\sum_{\alpha=1}^t
a_{\alpha}\Big(
Z_{\alpha}+Y^*\big(Z_{\alpha}^*\big)^{-1}Y
\Big)^{-1}\otimes
(T_{\alpha}T_{\alpha}^*),\quad
Q_{2,2}=\sum_{\alpha=1}^t
\overline{a}_{\alpha} 
\big( Z_{\alpha}^*+YZ_{\alpha}^{-1}Y^* \big)^{-1} \otimes
(T_{\alpha}T_{\alpha}^*),
\\
&Q_{1,2}=\sum_{\alpha=1}^t \overline{a}_{\alpha}
Z_{\alpha}^{-1}Y^* \big( Z_{\alpha}^*+YZ_{\alpha}^{-1}Y^* \big)^{-1}
\otimes
(T_{\alpha}T_{\alpha}^*),   
\\
&
Q_{2,1}=-\sum_{\alpha=1}^t a_{\alpha}\big(Z_{\alpha}^*\big)^{-1}
Y \Big(
Z_{\alpha}+Y^*\big(Z_{\alpha}^*\big)^{-1}Y
\Big)^{-1}\otimes
(T_{\alpha}T_{\alpha}^*).
\end{aligned}
\end{equation*} 
\end{small}

In view of \eqref{Zalpha plus inverse} and \eqref{Talpha Talphastar}, we have
\begin{small}
\begin{equation}\label{Q11}
\mathbb{I}_{n^2}+\sqrt{\gamma_N}Q_{1,1}=O(\Delta),
\end{equation}
\end{small}
where
\begin{small}
\begin{equation*}
\Delta=N^{-\frac{1}{2}}+\Omega_{\rm error}+\| Y \|^2.
\end{equation*}
\begin{equation}\label{Q22 plus I}
\mathbb{I}_{n^2}+\sqrt{\gamma_N}Q_{2,2}
=O(\Delta).
\end{equation} 
\end{small}
Also, we have 
\begin{small}
\begin{equation}\label{Q12}
\begin{aligned}
\sqrt{\gamma_N}Q_{1,2}
&=
\sum_{\alpha=1}^t
 \frac{\overline{a}_{\alpha}}{-a_{\alpha}}
 \frac{Y^*}{\overline{-a_{\alpha}}}\otimes
\Big( 
 \frac{\tau c_{\alpha}}{f_{\alpha}}\mathbb{I}_{n}
  \Big)+O(\Delta)
  \\
&=\tau\overline{P}_0(Y^*\otimes
\mathbb{I}_{n})+O(\Delta),
\end{aligned}
\end{equation}
\end{small}
and
\begin{small}
\begin{equation}\label{Q21}
\sqrt{\gamma_N}Q_{2,1}
=-\tau P_0(Y\otimes
\mathbb{I}_{n})+O(\Delta).
\end{equation}
\end{small}
 Combing \eqref{detMrewrite0}, \eqref{Q11}, \eqref{Q22 plus I}, \eqref{Q12} and \eqref{Q21} we have
\begin{small}
 \begin{equation*}
\begin{aligned}
&\det\!\bigg(
\widehat{L}_1+\sqrt{\gamma_N}\widehat{L}_2
\bigg)=\big| \det\big( A_{t+1} \big) \big|^{2n}\Big(
\prod_{\alpha=1}^t
f_{\alpha}^{n^2} +O\big(
N^{-\frac{1}{2}}+\| Y \|
\big)
\Big)   \\
&\times \left( \tau^{2n^2}|P_0|^{2n^2}
\big(\det(YY^*)\big)^{n}
 +O\Big(
\sum_{\alpha,\beta}\| Y \|^{\alpha}\big(
N^{-\frac{1}{2}}+\Omega_{\rm error}
\big)^{\beta}
\Big)
\right) ,
\end{aligned}
\end{equation*}
\end{small}
where the finite non-negative integer pair $(\alpha,\beta)$ is such that \eqref{M error condition 0}.

For the {\bf final proof}, computation of $e^{NF_0}D_{N,n}$ can be obtained directly by setting $r_0=0$ and $z_0=0$ in \eqref{eNf0 DNn}. 

By repeating almost the same procedure we can get the desired result. Last but not least, in case $z_0=0$, the change of integration variables is
\begin{equation*}
Y\rightarrow N^{-\frac{1}{4}}Y,\quad
\big(
G_1,\Gamma_1,\widetilde{T}_{{\rm d},\alpha},T_{{\rm u},\alpha}
\big)
\rightarrow
N^{-\frac{1}{2}}\big(
G_1,\Gamma_1,\widetilde{T}_{{\rm d},\alpha},T_{{\rm u},\alpha}
\big),\quad
\alpha=2,\cdots,t.
\end{equation*}
When comparing \eqref{f1 0} with \eqref{f1Non0}, we discover that 
$\Gamma_1$ takes place of $(Q_0Q_0^*)^{({\rm d})}$, while $G_1$ takes place of $(Q_0Q_0^*)^{({\rm off},{\rm u})}$, other terms remain unchanged. This is caused by the difference of typical sizes of $\Gamma_1$ and $G_1$ between $z_0=0$ and $z_0\not=0$. So when dealing with the term \eqref{triangle1 z0equal0}, we first need to integrate out 
$\{ T_{{\rm u},\alpha}|\alpha=2,\cdots,t \}$. Applying Proposition \ref{matrixintegral2} gets to the fact that the result does not depend on $G_1$. When dealing with the terms \eqref{diagonal1 z0equal0}, again we need to integrate out $\{ \widetilde{T}_{{\rm d},\alpha}|\alpha=2,\cdots,t \}$. We then get a similar result in form with \eqref{diagonal simplfy}.

%
%
%
 \hspace*{\fill}

\hspace*{\fill}

 \noindent{\bf Acknowledgements}  
 

 This  work 
 was   supported by  the National Natural Science Foundation of China \# 12371157 and \#12090012.

 
%
%
%

   \appendix
 \section{
 Properties on matrices and integrals} \label{Appendix}
 
 The  following well-known property for tensor product (see e.g. \cite[eqn(6)]{HS81}).
\begin{proposition}\label{tensorproperty}
For  a $p\times q$  matrix  A and an $m\times n$ matrix   B,  there exist  permutation matrices $\mathbb{I}_{p,m}$,  which only depend on  $p, m$ and satisfy  $\mathbb{I}_{p,m}^{-1}=\mathbb{I}_{m,p}=\mathbb{I}_{p,m}^t$, such that 

\begin{small}
\begin{equation}\label{tensorpropertyequation}
\mathbb{I}_{p,m}\left( A\otimes B \right)\mathbb{I}_{q,n}^{-1}=B\otimes A.
\end{equation}
\end{small}
  \end{proposition}

 An  integral  over non-positive definite matrices  will be useful, for proof see \cite[Proposition A.2]{LZ23}.
\begin{proposition}\label{matrixintegral1}
For  a  non-positive definite Hermitian matrix  $H_n=\left[ h_{i,j} \right]_{i,j=1}^n$,  given an integer  $ r_0\ge n$,  then 
\begin{small}
\begin{equation}\label{matrixintegral1equ}
\int_{H_n\leq 0} \left( \det(H_n) \right)^{r_0-n} \prod_{i<j}^n{\rm d}h_{i,j}
=
 \pi^{\frac{n(n-1)}{2}}  ( (r_0-1)! )^{-n} \prod_{k=1}^n   (r_0-k)!  \prod_{j=1}^n  (h_{j,j})^{r_0-1}.
\end{equation}
\end{small}
\end{proposition}

We also need a Gaussian integral over several strictly upper triangular matrices.
\begin{proposition}\label{matrixintegral2}
Let T, $\{ G_{\alpha}|\alpha=2,\cdots,t \}$ be strictly upper triangular matrices of size $n\times n$. Assume that $\sum_{\alpha=1}^t\frac{\tau c_{\alpha}}{f_{\alpha}}=1$, with $f_{\alpha}=|z_0-a_{\alpha}|^2$ and $P_0=\sum_{\alpha=1}^t\frac{c_{\alpha}(a_{\alpha}-z_0)}{f_{\alpha}^2}$, then
\begin{small}
\begin{equation}\label{L}
\begin{aligned}
&\int \cdots \int \exp\Big\{
-\frac{f_1^2}{\tau^2 c_1}{\rm Tr}\big(
T+\sum_{\alpha=2}^t \sqrt{\frac{\tau c_{\alpha}}{f_{\alpha}}} 
G_{\alpha}
\big)
\big(
T+\sum_{\alpha=2}^t \sqrt{\frac{\tau c_{\alpha}}{f_{\alpha}}} 
G_{\alpha}
\big)^*-\sum_{\alpha=2}^t\frac{f_{\alpha}}{\tau}{\rm Tr}\big(
G_{\alpha}G_{\alpha}^*
\big)           \\
&+\frac{1}{\tau}\sum_{1\leq i<j\leq n}\Big|
-(a_1-z_0)T_{i,j}+
\sum_{\alpha=2}^t
\sqrt{\frac{\tau c_{\alpha}}{f_{\alpha}}} 
(a_{\alpha}-a_1)(G_{\alpha})_{i,j}
\Big|^2\Big\}\prod_{\alpha=2}^t {\rm d}G_{\alpha}
\\&=\tau^{\frac{n(n-1)}{2}(t-2)}\Big(
\frac{f_1}{c_1}\prod_{\alpha=1}^t f_{\alpha}
|P_0|^2
\Big)^{-\frac{n(n-1)}{2}}
\pi^{\frac{n(n-1)(t-1)}{2}}.
\end{aligned}
\end{equation}
\end{small}
\end{proposition}
\begin{proof}
By letting $ a_{\alpha}\rightarrow \sqrt{\tau}a_{\alpha} $ and $ z_0\rightarrow \sqrt{\tau}z_0 $, without loss of generality, it suffices to prove the conclusion with $\tau=1.$

Denote the exponential term on the left hand side of \eqref{L} as $L(G_{\alpha},T)$. From basic computations, we can rewrite $L(G_{\alpha},T)$ as the polynomials of $(G_{\alpha})_{i,j}$ for $\alpha=2,\cdots,t,\ 1\leq i<j\leq n$,
\begin{small}
\begin{equation*}
\begin{aligned}
L(G_{\alpha},T)=\sum_{1\leq i<j\leq n}\Big(
-(\vec{G}_{i,j}-T_{i,j}\vec{d}A^{-1})A
(\vec{G}_{i,j}-T_{i,j}\vec{d}A^{-1})^*
+|T_{i,j}|^2\big(\vec{d}A^{-1}\vec{d}^*-\frac{f_1^2}{c_1}
+f_1\big)
\Big),
\end{aligned}
\end{equation*}
\end{small}
where
\begin{small}
\begin{equation*}
\begin{aligned}
&\vec{G}_{i,j}=\big( (G_{2})_{i,j},\cdots,(G_{t})_{i,j} \big),\quad
A=[A_{\alpha,\beta}]_{\alpha,\beta=2}^t,\quad
\vec{d}=\big( d_2,\cdots,d_t \big),       \\
&A_{\alpha,\beta}=\Big(\frac{f_1^2}{c_1}
-(a_{\alpha}-a_1)\overline{a_{\beta}-a_1}\Big)
\sqrt{\frac{c_{\alpha}c_{\beta}}{f_{\alpha}f_{\beta}}}+f_{\alpha}\delta_{\alpha,\beta},\quad
\alpha,\beta=2,\cdots,t,     \\
&d_{\alpha}=-\sqrt{\frac{c_{\alpha}}{f_{\alpha}}}
\Big( (a_1-z_0)\overline{a_{\alpha}-a_1}+\frac{f_1^2}{c_1} \Big),
\quad
\alpha=2,\cdots,t,
\end{aligned}
\end{equation*}
\end{small}
Calculate the gaussian integrals according to row vectors $\vec{G}_{i,j}$ and we obtain
\begin{small}
\begin{equation*}
\int e^{L(G_{\alpha},T)}\prod_{\alpha=2}^t {\rm d}G_{\alpha}=
\exp\Big\{
\big(\vec{d}A^{-1}\vec{d}^*-\frac{f_1^2}{c_1}
+f_1\big){\rm Tr}(TT^*)
\Big\}
\big( \det(A) \big)^{-\frac{n(n-1)}{2}}\pi^{\frac{n(n-1)}{2}(t-1)},
\end{equation*}
\end{small}
Later we will prove
\begin{small}
\begin{equation}\label{matrixintegral2 ke1}
\det(A)=\frac{f_1}{c_1}\Big(
\prod_{\alpha=1}^t f_{\alpha}
\Big)|P_0|^2
\end{equation}
\end{small}
and
\begin{small}
\begin{equation}\label{matrixintegral2 ke2}
\vec{d}A^{-1}\vec{d}^*-\frac{f_1^2}{c_1}
+f_1=0,
\end{equation}
\end{small}
these two facts complete the proof.

{\bf To prove\eqref{matrixintegral2 ke1}}, first rewrite A as
\begin{small}
\begin{equation*}
A={\rm diag}\Big(\sqrt{\frac{c_{\alpha}}{f_{\alpha}}}\Big)_{\alpha=2}^t
\Big( \frac{f_1^2}{c_1}\vec{1}^t\vec{1}+\widehat{A} \Big)
{\rm diag}\Big(\sqrt{\frac{c_{\beta}}{f_{\beta}}}\Big)_{\beta=2}^t,
\end{equation*}
\end{small}
where
\begin{small}
$$  
\vec{d}A^{-1}\vec{d}^*-\frac{f_1^2}{c_1}+f_1
 =\frac{f_1^2}{K_1 c_1}\Big(
\sum_{\alpha=1}^t
\frac{c_{\alpha}}{f_{\alpha}}-1
\Big)=0,
$$
\end{small}
Then 
\begin{small}
\begin{equation}\label{A rewrite 5}
\det(A)=\det\big( \widehat{A} \big)
\big( 1+\frac{f_1^2}{c_1}\vec{1}\widehat{A}^{-1}\vec{1}^t \big)
\prod_{\alpha=2}^t\frac{c_{\alpha}}{f_{\alpha}}
\end{equation}
\end{small}
From basic computations, for $\alpha,\beta=2,\cdots,t$,
\begin{small}
\begin{equation}\label{Ahat inverse}
\begin{aligned}
&\big(\widehat{A}^{-1}\big)_{\alpha,\beta}
=\frac{c_{\alpha}}{f_{\alpha}^2}\delta_{\alpha,\beta}+
\frac{c_{\alpha}c_{\beta}}{f_{\alpha}^2f_{\beta}^2}
(a_{\alpha}-a_1)\overline{a_{\beta}-a_1}
\Big( 1-\sum_{\alpha=2}^t
\frac{c_{\alpha}}{f_{\alpha}^2}|a_{\alpha}-a_1|^2 \Big)^{-1}
,\\&
\det\big( \widehat{A} \big)=\prod_{\alpha=2}^t
\frac{f_{\alpha}^2}{c_{\alpha}}
\Big( 1-\sum_{\alpha=2}^t
\frac{c_{\alpha}}{f_{\alpha}^2}|a_{\alpha}-a_1|^2 \Big).
\end{aligned}
\end{equation}
\end{small}
Recall the definition of $P_0$ and $P_1$ in \eqref{parameter} and \eqref{parameter2},
\begin{small}
\begin{equation*}
\Big|
\sum_{\alpha=2}^t
\frac{c_{\alpha}}{f_{\alpha}^2}(a_{\alpha}-a_1)
\Big|^2
=\Big|
\sum_{\alpha=2}^t
\frac{c_{\alpha}}{f_{\alpha}^2}\big((a_{\alpha}-z_0)
-(a_{1}-z_0)
\big)
\Big|^2=
|P_0-(a_1-z_0)P_1|^2.
\end{equation*}
\end{small}
Also from $P_{00}(z_0)=\sum_{\alpha=1}^t
\frac{c_{\alpha}}{f_{\alpha}}=1,$ 
\begin{small}
\begin{equation*}
\begin{aligned}
\sum_{\alpha=2}^t
\frac{c_{\alpha}}
{f_{\alpha}^2}|a_{\alpha}-a_1|^2&=
\sum_{\alpha=1}^t
\frac{c_{\alpha}}{f_{\alpha}^2}\big|(a_{\alpha}-z_0)
-(a_{1}-z_0)
\big|^2     \\
&=1-\overline{a_1-z_0}P_0
-(a_1-z_0)\overline{P}_0+|a_1-z_0|^2P_1.
\end{aligned}
\end{equation*}
\end{small}
so from \eqref{Ahat inverse} and basic computations we obtain
\begin{small}
\begin{equation*}
\big( 1+\frac{f_1^2}{c_1}\vec{1}\widehat{A}^{-1}\vec{1}^t \big)
\Big( 1-\sum_{\alpha=2}^t
\frac{c_{\alpha}}{f_{\alpha}^2}|a_{\alpha}-a_1|^2 \Big)
=\frac{f_1^2}{c_1}|P_0|^2
\end{equation*}
\end{small}
Hence,
Combining \eqref{A rewrite 5} and \eqref{Ahat inverse} we obtain \eqref{matrixintegral2 ke1}.

{\bf To prove \eqref{matrixintegral2 ke2}}, first we have
\begin{small}
\begin{equation}\label{Ainverse}
\big( A^{-1} \big)_{\alpha,\beta}=
\sqrt{\frac{f_{\alpha}f_{\beta}}{c_{\alpha}c_{\beta}}}
\frac{c_{\alpha}c_{\beta}}{f_{\alpha}^2 f_{\beta}^2 K_1}
\big(\frac{f_{\alpha}^2}{c_{\alpha}}\delta_{\alpha,\beta}K_1
+K_{\alpha,\beta}\big),
\end{equation}
\end{small}
with 
\begin{small}
\begin{equation}\label{K1}
K_1=1+\frac{f_1^2}{c_1}G_2
-G_1\big(1+\frac{f_1^2}{c_1}G_2\big)+\frac{f_1^2}{c_1}|G_3|^2,
\end{equation}
\end{small}
and 
\begin{small}
\begin{equation}\label{Kalphabeta}
K_{\alpha,\beta}=(a_{\alpha}-a_1)\overline{a_{\beta}-a_1}
-\frac{f_1^2}{c_1}+\frac{f_1^2}{c_1}G_1-\frac{f_1^2}{c_1}
\overline{a_{\beta}-a_1}G_3
-\frac{f_1^2}{c_1}(a_{\alpha}-a_1)\overline{G}_3
+\frac{f_1^2}{c_1}(a_{\alpha}-a_1)\overline{a_{\beta}-a_1}G_2,
\end{equation}
\end{small}
where
\begin{small}
\begin{equation*}
G_1=\sum_{\alpha=2}^t
\frac{c_{\alpha}}{f_{\alpha}^2}|a_{\alpha}-a_1|^2,\quad
G_2=\sum_{\alpha=2}^t
\frac{c_{\alpha}}{f_{\alpha}^2},\quad
G_3=\sum_{\alpha=2}^t
\frac{c_{\alpha}}{f_{\alpha}^2}(a_{\alpha}-a_1).
\end{equation*}
\end{small}
\eqref{Ainverse} can be seen by verifying the equality
 \begin{small}
 $$
 \sum_{\gamma=2}^t A_{\alpha,\gamma}
\big(
A^{-1}
\big)_{\gamma,\beta}=\delta_{\alpha,\beta}
 $$
 \end{small}
directly. To compute $ \vec{d}A^{-1}\vec{d}^*=
\sum_{\alpha,\beta=2}^t d_{\alpha}
\big( A^{-1} \big)_{\alpha,\beta}
d_{\beta} $, we first expand the right hand side according to the terms of  $K_1$ and $K_{\alpha,\beta}$ in \eqref{K1} and \eqref{Kalphabeta}, then for each term we compute the sum $\sum_{\alpha,\beta=2}^t$ to the form represented by $G_1$, $G_2$ and $G_3$, for example, for the first term $(a_{\alpha}-a_1)\overline{a_{\beta}-a_1}$ in $K_{\alpha,\beta}$(cf \eqref{Kalphabeta}), we have
\begin{small}
\begin{equation*}
\begin{aligned}
&\sum_{\alpha,\beta=2}^t 
\frac{c_{\alpha}c_{\beta}}{f_{\alpha}^2f_{\beta}^2}
(a_{\alpha}-a_1)\overline{a_{\beta}-a_1}
\big( (a_1-z_0)\overline{a_{\alpha}-a_1}+\frac{f_1^2}{c_1} \big)
\big( \overline{a_1-z_0}(a_{\beta}-a_1)+\frac{f_1^2}{c_1} \big)
      \\
&=\big( (a_1-z_0)G_1+\frac{f_1^2}{c_1}G_3 \big)
\big( \overline{a_1-z_0}G_1+D\frac{f_1^2}{c_1}\overline{G}_3 \big),
\end{aligned}
\end{equation*}
\end{small}
Following this process, after basic calculations we obtain
\begin{small}
\begin{equation}\label{dAinvers rewrite}
\begin{aligned}
\vec{d}A^{-1}\vec{d}^*-\frac{f_1^2}{c_1}+f_1
 =\frac{f_1^2}{K_1 c_1}\Big(
\frac{c_1}{f_1}+f_1G_2-1+G_1+(a_1-z_0)\overline{G}_3+\overline{a_1-z_0}G_3
\Big).
\end{aligned}
\end{equation}
\end{small}
Now rewrite $a_{\alpha}-a_1$ as $(a_{\alpha}-z_0)-(a_1-z_0) $, then we have
\begin{small}
\begin{equation*}
\begin{aligned}
&G_1=\sum_{\alpha=2}^t
\frac{c_{\alpha}}{f_{\alpha}}-\overline{a_1-z_0}\sum_{\alpha=2}^t
\frac{c_{\alpha}}{f_{\alpha}^2}(a_{\alpha}-z_0)
-(a_1-z_0)\sum_{\alpha=2}^t
\frac{c_{\alpha}}{f_{\alpha}^2}\overline{a_{\alpha}-z_0}
+f_1\sum_{\alpha=2}^t
\frac{c_{\alpha}}{f_{\alpha}^2}
\\
& G_2=\sum_{\alpha=2}^t
\frac{c_{\alpha}}{f_{\alpha}^2},\quad
G_3=\sum_{\alpha=2}^t
\frac{c_{\alpha}}{f_{\alpha}^2}(a_{\alpha}-z_0)
-(a_1-z_0)\sum_{\alpha=2}^t
\frac{c_{\alpha}}{f_{\alpha}^2}.
\end{aligned}
\end{equation*} 
\end{small}
Substituting the above into \eqref{dAinvers rewrite}, we get that
\begin{small}
$$  
\vec{d}A^{-1}\vec{d}^*-\frac{f_1^2}{c_1}+f_1
 =\frac{f_1^2}{K_1 c_1}\Big(
\sum_{\alpha=1}^t
\frac{c_{\alpha}}{f_{\alpha}}-1
\Big)=0,
$$
\end{small}
and finally we get \eqref{matrixintegral2 ke2}.
\end{proof}

\end{document}